\documentclass[11pt,reqno]{amsart}
\usepackage[T1]{fontenc}
\usepackage{lmodern,amsmath,amssymb,amsthm,mathtools,mathrsfs,booktabs,microtype}
\usepackage[margin=1.1in]{geometry}
\usepackage{xcolor}
\usepackage[hidelinks]{hyperref}
\hypersetup{colorlinks,	linkcolor={red!50!black},citecolor={blue!50!black},	urlcolor={blue!80!black},
	pdftitle={Ramified convolution and monodromy bases for symmetrizable Kac-Moody algebras},pdfauthor={Jiarui Fei}}

\newtheorem{theorem}{Theorem}[section]
\newtheorem{proposition}[theorem]{Proposition}
\newtheorem{lemma}[theorem]{Lemma}
\newtheorem{corollary}[theorem]{Corollary}
\newtheorem{maintheorem}{Theorem}

\theoremstyle{definition}

\newtheorem{remark}[theorem]{Remark}
\newtheorem{example}[theorem]{Example}

\newcommand{\op}[1]{\operatorname{#1}}
\newcommand{\mr}[1]{\mathrm{#1}}
\newcommand{\mb}[1]{\mathbb{#1}}
\newcommand{\mc}[1]{\mathcal{#1}}
\newcommand{\mf}[1]{\mathfrak{#1}}
\newcommand{\ms}[1]{\mathscr{#1}}
\newcommand{\C}{\mb{C}}
\newcommand{\Q}{\mb{Q}}
\newcommand{\Z}{\mb{Z}}
\newcommand{\A}{\mb{A}}
\newcommand{\OO}{\mc{O}}
\newcommand{\PP}{\ms{P}}
\newcommand{\Bcal}{\mc{B}}
\newcommand{\Fl}{\op{Fl}}
\newcommand{\Gr}{\op{Gr}}
\newcommand{\Res}{\op{Res}}
\newcommand{\Tr}{\op{Tr}}
\newcommand{\rk}{\op{rank}}
\newcommand{\Hom}{\op{Hom}}
\newcommand{\End}{\op{End}}
\newcommand{\Ext}{\op{Ext}}
\newcommand{\im}{\op{im}}
\newcommand{\coker}{\op{coker}}
\newcommand{\fac}{\op{fac}}
\newcommand{\lf}{\mr{lf}}

\newcommand{\mon}{\mr{mon}}

\newcommand{\ev}{\op{ev}}
\newcommand{\lcm}{\op{lcm}}
\newcommand{\Gm}{\mb{G}_m}

\numberwithin{equation}{section}
\title[Ramified convolution and monodromy bases]{Ramified convolution and monodromy bases\newline for symmetrizable Kac--Moody algebras}
\author{Jiarui Fei}
\address{School of Mathematical Sciences, Shanghai Jiao Tong University, China}
\email{jiarui@sjtu.edu.cn}
\thanks{\hspace{-.17in}{The author was supported in part by National Natural Science Foundation of China (Nos.~12131015 and 12571038)}}
\date{}

\begin{document}
\begin{abstract}
We realize the positive enveloping algebra of a symmetrizable Kac--Moody algebra by convolution on a ramified preprojective species over $\C((t))$. The coefficients are alternating traces of tame monodromy on representation flag varieties. We prove the relative Serre relations modulo classes induced from proper finite extensions. Monodromy evaluations at compatible generic points of the central reductions yield an integral basis respecting the top and socle crystals. The basis is unchanged by scaling the symmetrizer, induces bases of integrable representation lattices, and agrees with Lusztig's semicanonical basis when the Cartan matrix is symmetric. We prove the acyclic convolution conjecture of Gei\ss--Leclerc--Schr\"oer for all symmetrizable Cartan matrices and all positive integral symmetrizers. Finally, we give a counterexample to the Support Conjecture of Gei\ss--Leclerc--Schr\"oer and compute a ramified lift for which monodromy restores the Serre relation.
\end{abstract}

\maketitle
\tableofcontents

\section*{Introduction}\label{sec:introduction}

\subsection{Historical background}
The geometric study of the positive part of a Kac--Moody enveloping algebra through quiver representations began with Lusztig's work on convolution \cite{Lusztig1991}. Let $C=(c_{ij})_{i,j\in I}$ be a symmetrizable generalized Cartan matrix, where $I$ is finite. We write $\mf{n}(C)$ for the positive nilpotent subalgebra of the associated Lie algebra, $e_i$ for its Chevalley generators, and $U^+_\Q=U_\Q(\mf{n}(C))$. For symmetric $C$, Kashiwara and Saito gave a geometric realization of the crystal using nilpotent preprojective representations \cite{KashiwaraSaito}, and Lusztig constructed the \emph{semicanonical basis} from the irreducible components of their varieties \cite{Lusztig}. Gei\ss, Leclerc and Schr\"oer (GLS) studied the dual semicanonical basis and its multiplication \cite{GLS2005}. They subsequently related generating functions of Euler characteristics of partial composition series to cluster characters and generic bases associated with unipotent cells \cite{GLSJAMS}.

The extension of this geometry to nonsymmetric Cartan matrices requires a symmetrizer. Fix positive integers $d=(d_i)_{i\in I}$ with $d_ic_{ij}=d_jc_{ji}$, and put $D=\op{diag}(d_i)$. For an acyclic orientation $\Omega$ of the underlying graph, GLS introduced the algebras $H(C,D,\Omega)$ and $\Pi(C,D)$ in \cite{GLSI}. The subsequent papers treat changes of symmetrizers \cite{GLSII}, convolution algebras \cite{GLSIII}, and nilpotent varieties, crystal structures, and semicanonical functions \cite{GLSIV}. The dimension theorem, component count, and crystal on maximal components established there are basic ingredients of this paper. To obtain a basis from generic Euler-characteristic evaluations, GLS proposed a support condition on the Serre relations \cite[Conjecture~1.4]{GLSIV}. We follow the terminology ``Support Conjecture'' used in \cite[Proposition~5.11]{GeissICM}.

A power-series construction of hereditary species with the same ramified vertex fields and equivalent arrow bimodules appears in \cite[Section~4]{GLSRS}. That paper also studies rigid modules, that is, modules without self-extensions, and their relation to Schur roots. Nandakumar and Tingley obtained another geometric realization of symmetrizable crystals through modulated graphs \cite{NandakumarTingley}.

In earlier work \cite{FeiQP}, we introduced monodromy characters for ramified quivers with potentials by taking alternating traces of tame inertia on quiver Grassmannians. That paper establishes mutation formulas and, under its stated hypotheses, cluster bases. The present paper concerns preprojective representations and their flags instead. We use the coefficient-field comparison, the trace calculation for cyclically permuted strata, and the ramified Grassmannian formula of \cite[Proposition~7.2 and Lemmas~7.3--7.4]{FeiQP}. These results allow us to replace ordinary Euler-characteristic evaluations by monodromy traces in the preprojective construction.

\subsection{Ramification and monodromy}
The need for monodromy is visible already in the Grassmannians at a single vertex. We put
\[ \OO=\C[[t]],\qquad K=\C((t)),\qquad K_i=\C((t^{1/d_i})), \]
and choose an algebraic closure $\overline K$ containing compatible roots of $t$. A representation of a species has a $K_i$-vector space at vertex $i$. Over $\overline K$, the field $K_i$ splits into $d_i$ copies of $\overline K$. This changes the Euler characteristics of the relevant Grassmannians and accounts for the difference between ordinary Euler characteristic and the invariant we use.

Fix a prime $\ell$ and a topological generator $\sigma$ of the tame inertia group of $K$. For a separated scheme $X$ of finite type over $K$, define its \emph{monodromy trace} by
\[ \chi_\sigma(X)=\sum_q(-1)^q\Tr\bigl(\sigma\mid H_c^q(X_{\overline K},\Q_\ell)\bigr), \]
where $H_c^q$ denotes compactly supported $\ell$-adic cohomology. We write $\chi_c$ for the corresponding alternating sum of cohomology dimensions, or for the usual compactly supported Euler characteristic of a complex variety. By Nicaise's trace formula, $\chi_\sigma(X)$ is an integer independent of $\ell$ and of the chosen generator \cite[Corollary~6.5]{Nicaise}.

For a finite extension $L/K$, let $\Res_{L/K}$ denote \emph{Weil restriction}, and write $\Gr_a(V)$ for the Grassmannian of $a$-dimensional subspaces of $V$. For $0\le a\le r$ we have
\[ \chi_c\bigl((\Res_{K_i/K}\Gr_a(K_i^r))_{\overline K}\bigr)=\binom ra^{d_i},\qquad \chi_\sigma\bigl(\Res_{K_i/K}\Gr_a(K_i^r)\bigr)=\binom ra. \]
Indeed, after passing to $\overline K$ the Weil restriction becomes a product of $d_i$ Grassmannians, and inertia cyclically permutes its factors. The second equality follows from the K\"unneth formula; see \cite[Lemma~7.4]{FeiQP}. Representation flag varieties are subvarieties of similar products, defined by the conditions that the arrows preserve the flags. Their monodromy traces will give the Serre relations.

We use the vertex valuation rings $\OO_i=\OO[x_i]/(x_i^{d_i}-t)$ and the \emph{standard integral preprojective algebra} $\PP_d$ defined in Section~\ref{sec:setup}. Its generic fibre is $\Gamma_d=K\otimes_\OO\PP_d$. We work with right modules and multiply paths from left to right. For $k\ge1$, its central reduction is
\[ \Pi_k:=\PP_d/t^k\PP_d\cong\Pi(C,kD)^{\mr{op}}, \]
where $\Pi(C,kD)$ is the GLS preprojective algebra. We call a $\PP_d$-module a \emph{locally free lattice} if its vertex modules are free of finite rank over the $\OO_i$.

\subsection{Main results}
Let $J$ be the ideal of $\Gamma_d$ generated by its nonloop arrows. A finite-dimensional right $\Gamma_d$-module $M$ is \emph{nilpotent} if $MJ^n=0$ for some $n$. The module $S_i$ has vertex space $K_i$ at $i$ and zero spaces elsewhere. For a word $w=(i_1,\ldots,i_h)$ in $I$, let $\Fl^R_w(M)$ be the projective scheme of ascending submodule flags with successive quotients $S_{i_1},\ldots,S_{i_h}$. In families, each quotient has rank one over its vertex field, a condition we call \emph{balanced}. Set
\[ F_w(M)=\chi_\sigma(\Fl^R_w(M)), \]
with value zero when the word has the wrong rank and $F_\varnothing(0)=1$. Write $\mc{M}_\sigma(\Gamma_d)$ for the $\Q$-span of these functions on nilpotent modules.

\begin{maintheorem}[Ramified convolution]\label{intro:algebra}
For every symmetrizable Cartan matrix $C$ and every positive integral symmetrizer $D=\op{diag}(d_i)$, geometric convolution makes $\mc{M}_\sigma(\Gamma_d)$ a graded connected cocommutative Hopf algebra, and the assignment $e_i\mapsto F_i$ induces an isomorphism
\[ U^+_\Q\xrightarrow{\sim}\mc{M}_\sigma(\Gamma_d). \]
The coproduct is given by evaluation on direct sums of representations.
\end{maintheorem}

The vertex ring of $\Pi_k$ at $i$ is $H_{i,k}=\C[x_i]/(x_i^{kd_i})$. Let $E_{i,k}$ be the module supported at $i$ with vertex module $H_{i,k}$ and with all nonloop arrows acting as zero. A module is \emph{locally free} if its vertex modules are free over these rings; it is \emph{$E$-filtered} if it has a filtration with successive quotients among the $E_{i,k}$. For a rank vector $\mathbf r=(r_i)\in\mb{N}^I$, where $\mb{N}=\{0,1,2,\ldots\}$, let $\Bcal_{kd}(\mathbf r)$ be the set of irreducible components of largest dimension in the closure of the $E$-filtered locus on the fixed vertex modules $H_{i,k}^{r_i}$. For $k=1$ we simply write $\Bcal_d(\mathbf r)$, and put
\[ q_d(\mathbf r)=\tfrac12\mathbf r^{\mathsf T}DC\mathbf r,\qquad N_d(\mathbf r)=\sum_i d_ir_i^2-q_d(\mathbf r). \]
The largest quotient and submodule supported at $i$ are called its \emph{$i$-top} and \emph{$i$-socle}; the corresponding operations on components give the two crystal structures of GLS.

Let $U^+_\Z$ be the \emph{divided-power integral form} of $U^+_\Q$, generated by $e_i^{(p)}=e_i^p/p!$. A subscript $\mathbf r$ indicates the homogeneous part of rank $\mathbf r$, with $\deg e_i=\alpha_i$, the $i$th coordinate vector of $\mb{Z}^I$.

\begin{maintheorem}[Generic monodromy basis]\label{intro:basis}
Reduction modulo $t^k$ induces compatible bijections
\[ \Bcal_{kd}(\mathbf r)\xrightarrow{\sim}\Bcal_d(\mathbf r),\qquad k\ge1, \]
which preserve the two crystal structures. For components corresponding under these bijections, reduction from level $k$ to level $l\le k$ is dominant. Their compatible generic points determine functionals $\lambda_Z^{\mon}$, independent of the choices used in their construction, and evaluation is an isomorphism
\[ U^+_{\Z,\mathbf r}\xrightarrow{\sim}\Z^{\Bcal_d(\mathbf r)},\qquad u\longmapsto\bigl(\lambda_Z^{\mon}(u)\bigr)_Z. \]
Consequently there is an integral basis $\{b_Z^{\mon}\}$ characterized by $\lambda_Y^{\mon}(b_Z^{\mon})=\delta_{Y,Z}$.
\end{maintheorem}

The basis is compatible with both crystal structures and gives bases of integrable representation lattices (Theorem~\ref{thm:perfect} and Proposition~\ref{prop:integralstrings}). It is invariant under common scaling of the symmetrizer and under the coordinate changes considered in Section~\ref{sec:invariance}. For symmetric $C$ it recovers Lusztig's semicanonical basis (Proposition~\ref{prop:symmetric}). The multiplication formula for its dual holds when the two generic cross-$\Ext^1$ spaces vanish (Proposition~\ref{prop:Ext-product}).

The classical special-fibre convolution has a different behaviour. Let $T_\Q=\Q\langle z_i:i\in I\rangle$ be the free associative algebra, and let $J_C$ be the two-sided ideal generated by the Serre elements $(\op{ad}z_i)^{1-c_{ij}}z_j$, where $(\op{ad}x)(y)=xy-yx$. The quotient $T_\Q/J_C$ is $U^+_\Q$. A monodromy functional is therefore also a functional on $T_\Q$. For a maximal component $Z$ of $\Pi_1$, let $\lambda_Z^{\mr{sp}}$ denote the generic value of the corresponding special-fibre convolution functions, before imposing the Serre relations. The \emph{Support Conjecture} of \cite[Conjecture~1.4]{GLSIV} requires every function in the Serre ideal to have support of dimension strictly below the maximal-component dimension in its degree.

\begin{maintheorem}[A counterexample to the Support Conjecture]\label{intro:counterexample}
For
\[ C_*=\begin{pmatrix}2&-2&0\\-1&2&-1\\0&-2&2\end{pmatrix},\qquad d_*=(1,2,1),\quad D_*=\op{diag}(d_*), \]
there is a maximal $E$-filtered component $Z$ of rank $(2,3,2)$ such that
\[ f=(\op{ad}z_2)^2(z_1)\,z_1z_3z_2z_3\in J_{C_*},\qquad \lambda_Z^{\mr{sp}}(f)=-2,\quad \lambda_Z^{\mon}(f)=0. \]
This component contains a dense orbit of a rigid module. In particular, the Support Conjecture fails even for the minimal symmetrizer.
\end{maintheorem}

The proof is an explicit calculation of three flag varieties and the endomorphism algebra of a module. For its ramified lift, the corresponding relative flag scheme explains why special-fibre Euler characteristic and monodromy trace give different answers.

There is also an acyclic realization. If $\Omega$ is an acyclic orientation, let $H=H(C,D,\Omega)^{\mr{op}}$, and let $\Lambda_d$ be the corresponding acyclic quotient of $\Gamma_d$. With $\mc{M}_\Q(H)$ denoting the Euler-characteristic convolution algebra generated by the vertex modules, we obtain the following result.

\begin{maintheorem}[Acyclic convolution and the GLS conjecture]\label{intro:acyclic}
For every symmetrizable Cartan matrix $C$, every positive integral symmetrizer $D$, and every acyclic orientation $\Omega$, the vertex-generator maps yield graded Hopf algebra isomorphisms
\[ U^+_\Q\xrightarrow{\sim}\mc{M}_\Q(H),\qquad U^+_\Q\xrightarrow{\sim}\mc{M}_\sigma(\Lambda_d). \]
The first isomorphism proves \cite[Conjecture~7.2]{GLSIII}, including the corresponding identification of primitive Lie algebras.
\end{maintheorem}

The Dynkin case of the first isomorphism was established in \cite[Theorem~7.1]{GLSIII}. The proof of the general case in Appendix~\ref{sec:acyclic} uses an auxiliary acyclic quiver and does not depend on the component charts. The same identification gives the canonical Serre quotient of the ordinary preprojective convolution algebra: restriction to any acyclic orientation has kernel exactly the Serre ideal (Theorem~\ref{thm:canonical-quotient}).

\subsection{Proof and organization}
To prove the Serre relations, we retain the flag data on both sides of the subquotient where a Serre relation is applied. On a locus where the relevant linear maps have constant rank, the remaining flag variety is a Grassmann bundle with two complete-flag refinements. The Galois action on its cohomology reduces the Serre expression to an alternating binomial sum modulo properly induced classes.

For the basis, we construct integral families using the kernels of the incoming maps. Compatible families give generic representations at every central length, and successive maximal-top removals produce a unitriangular evaluation matrix. This proves both the integral basis theorem and faithfulness of convolution.

Sections~\ref{sec:setup}--\ref{sec:serre} give the algebra and the relative Serre theorem; Section~\ref{sec:hopf} constructs convolution. Sections~\ref{sec:charts}--\ref{sec:basis} establish the component correspondence and the basis, including a calculation in type $B_2$. Sections~\ref{sec:perfect}--\ref{sec:applications} establish its structural properties. Section~\ref{sec:comparison} identifies the canonical Serre quotient and contrasts it with ordinary generic evaluation through the affine counterexample and its ramified lift. The appendix proves the acyclic realization.

\section{The standard integral algebra}\label{sec:setup}
We keep the Cartan matrix $C$, the symmetrizer $D=\op{diag}(d_i)$, and the fields and rings from Section~\ref{sec:introduction}. Choose an acyclic orientation $\Omega$ of the underlying graph. For adjacent vertices set $\epsilon_{ij}=1$ if $i\to j$ belongs to $\Omega$, and $\epsilon_{ji}=-\epsilon_{ij}$. Put
\[ g_{ij}=\gcd(|c_{ij}|,|c_{ji}|),\qquad f_{ij}=|c_{ij}|/g_{ij},\qquad e_{ij}=\gcd(d_i,d_j). \]
These numbers are used only when $c_{ij}<0$. Symmetrizability gives
\[ d_i=e_{ij}f_{ji},\qquad d_j=e_{ij}f_{ij}. \]

Take a quiver with vertex set $I$, a loop $x_i$ at each vertex, and $g_{ij}$ arrows $a_{ij}^{(c)}:i\to j$ for every ordered adjacent pair, where $1\le c\le g_{ij}$. Let $1_i$ be its vertex idempotents. We define $\PP_d$ as the quotient of its path algebra over $\OO$ by
\begin{align}
 &x_i^{d_i}=t1_i,\notag\\
 &x_i^{f_{ji}}a_{ij}^{(c)}=a_{ij}^{(c)}x_j^{f_{ij}},\label{eq:comm}\\
 \rho_i:=&\sum_{j:\,c_{ij}<0}\epsilon_{ij}\sum_{c=1}^{g_{ij}}\sum_{q=0}^{f_{ji}-1}
 x_i^qa_{ij}^{(c)}a_{ji}^{(c)}x_i^{f_{ji}-1-q}=0.\notag
\end{align}
We call the last relations the \emph{moment relations}. The right-module convention gives $a_{ij}^{(c)}\in1_i\PP_d1_j$, with action $M_i=M1_i\to M_j$, $v\mapsto va_{ij}^{(c)}$. The path $ab$ first follows $a$ and then $b$. When we write matrices of linear maps, we use column vectors: the matrix of right multiplication by $ab$ is the matrix of $b$ times the matrix of $a$.

These are the opposite GLS relations after relabelling arrows. The equivalence from GLS left modules to our right modules is $m\cdot a^{\mr{op}}=am$. It preserves submodules, so ascending flag words have the same order on the two sides. The generic and special algebras are
\[ \Gamma_d=K\otimes_\OO\PP_d,\qquad \Pi_k=\PP_d/t^k\PP_d\cong\Pi(C,kD)^{\mr{op}}. \]
A locally free lattice is a right $\PP_d$-module whose space at each vertex $i$ is finite free over $\OO_i$.

\subsection{Modules and their flags}
For a module $M$ over $\Gamma_d$, a locally free module over $\Pi_k$, or a locally free lattice, $\underline{\rk}M=(r_i(M))_{i\in I}$ denotes the ranks over the corresponding vertex fields or rings. We write $|\mathbf r|=\sum_i r_i$. For a lattice or a $\Pi_k$-module, nilpotence means annihilation by some power of the ideal generated by the nonloop arrows, as for $\Gamma_d$. The vertex parameters are not included in that ideal.

We use the simple generic modules $S_i$ and the special modules $E_{i,k}$ defined in Section~\ref{sec:introduction}; put $E_i=E_{i,1}$. An integral $E$-filtration has successive quotients supported at a single vertex and free of rank one over its vertex ring, with zero nonloop action. Every such filtration gives a nilpotent generic fibre.

For a word $\mathbf w=(i_1,\ldots,i_h)$ of rank $\sum_a\alpha_{i_a}$, the balanced flag scheme $\Fl^R_{\mathbf w}(M)$ has the following functorial meaning. On a $K$-scheme $Y$, it parametrizes flags
\[ 0=M^0\subset M^1\subset\cdots\subset M^h=M\otimes_K\mc{O}_Y,\qquad \underline{\rk}(M^a/M^{a-1})=\alpha_{i_a}, \]
where all members are arrow-stable and, at each vertex, are subbundles over $K_i\otimes_K\mc{O}_Y$ with locally free quotients. Thus the specified ranks hold at every geometric point, including at each split vertex. The superscript $R$ refers to $R=\prod_iK_i$. Arrow stability defines a closed subscheme of the product of the vertex flag varieties obtained by Weil restriction. In particular, $\Fl^R_{\mathbf w}(M)$ is projective over $K$.

We similarly write $\Fl^{\lf}_{\mathbf w}(M)$ for special-fibre flags with successive quotients $E_{i_1,k},\ldots,E_{i_h,k}$; the superscript $\lf$ means locally free. These flag schemes need not be proper. Their compactly supported Euler characteristics give the special-fibre convolution words. Our convolution convention evaluates the left factor on a submodule, so word order records successive quotients in ascending order. A \emph{partial flag} allows a successive quotient $S_i^p$, or $E_{i,k}^p$, in place of $p$ consecutive quotients of rank one; exponent zero means that the quotient is omitted.

The grading on the free algebra $T_\Q$ is $\deg z_i=\alpha_i$. We use the linear evaluation map
\[ \ev_\sigma:T_\Q\longrightarrow\mc{M}_\sigma(\Gamma_d),\qquad z_{i_1}\cdots z_{i_h}\longmapsto F_{\mathbf w}. \]
Section~\ref{sec:serre} proves that it annihilates the Serre ideal $J_C$. We use the same word notation for the Euler-characteristic evaluations at a central reduction.

\subsection{The component geometry used below}
With $d$ fixed, abbreviate $\Bcal_{kd}(\mathbf r)$ to $\Bcal_k(\mathbf r)$. These are the maximal components in the closure of the $E$-filtered locus on the fixed vertex modules $H_{i,k}^{r_i}$. We take their reduced structures when forming function fields. The vertex change-of-basis group is $\prod_i\op{GL}_{r_i}(H_{i,k})$, of complex dimension $k\sum_i d_ir_i^2$.

For a module $M$, write $\op{sub}_i(M)$ for its largest submodule supported at $i$, and $\fac_i(M)$ for its largest quotient supported at $i$. Thus these are the $i$-socle and $i$-top introduced above. Put $\mathsf K_i(M)=\ker(M\to\fac_i(M))$. The two \emph{canonical removals} are $\mathsf K_i(M)$ and $M/\op{sub}_i(M)$. Following \cite[Section~4.2]{GLSIV}, an $E$-filtered module is a \emph{crystal module} if all its $i$-tops and $i$-socles are locally free and both canonical removals are recursively crystal. The zero module starts the recursion. The supported tops and socles need not be semisimple.

The GLS dimension, generic-crystallinity, and component-count theorems give, for $Z\in\Bcal_k(\mathbf r)$,
\begin{equation}\label{eq:published}
 \dim Z=kN_d(\mathbf r),\qquad
 |\Bcal_k(\mathbf r)|=\dim_\Q U^+_{\Q,\mathbf r},\qquad
 N_d(\mathbf r)=\sum_{i<j}d_i(-c_{ij})r_ir_j.
\end{equation}
Here $i<j$ uses any fixed ordering of $I$; the summand is symmetric in the two endpoints. A dense open subset of every maximal component consists of crystal modules. We use \cite[Theorem~4.1, Proposition~4.4, and Theorems~5.4 and~7.2]{GLSIV} for these statements.

We also use the following maximal-top correspondence. If the generic free $i$-top of a maximal component has rank $p$, removing that entire top gives a maximal component of rank $\mathbf r-p\alpha_i$ whose generic $i$-top is zero. Conversely, the associated extension stratum has irreducible fibres over that smaller component. There is an analogous correspondence for socles. For tops, see \cite[Sections~3.7--3.9 and~5.3, Corollary~4.6]{GLSIV}; socles follow by duality. Repeated extensions by $E_{i,k}$ give the two crystal operators, with the geometric convention that adding $E_{i,k}$ raises the rank by $\alpha_i$.

\subsection{Incoming and outgoing maps}\label{sec:adjunction}
For an adjacent pair $i,j$, let
\[ \OO_{ij}=\OO[z]/(z^{e_{ij}}-t),\qquad z\longmapsto x_i^{f_{ji}}\text{ in }\OO_i,\quad z\longmapsto x_j^{f_{ij}}\text{ in }\OO_j. \]
The power basis $1,x_i,\ldots,x_i^{f_{ji}-1}$ makes $\OO_i$ free over $\OO_{ij}$. Define
\[ \ell_{ij}:\OO_i\longrightarrow\OO_{ij},\qquad \ell_{ij}\!\left(\sum_{q=0}^{f_{ji}-1}c_qx_i^q\right)=c_{f_{ji}-1},\quad c_q\in\OO_{ij}. \]
This is a \emph{Frobenius functional}: the pairing $(u,v)\mapsto\ell_{ij}(uv)$ is perfect over $\OO_{ij}$. Its matrix in the power basis has ones on the anti-diagonal and zeros elsewhere. Indeed, a power $x_i^{p+q}$ with $p+q\ge f_{ji}$ reduces to a multiple of a power of exponent at most $f_{ji}-2$. The pairing remains perfect after every base change. It gives the integral adjunction underlying \cite[Section~5]{GLSI} and \cite[Section~2.3]{GLSIV}.

For a locally free representation $U$ of $\PP_d$, set
\[ W_i(U)=\bigoplus_{j:\,c_{ij}<0}\ \bigoplus_{c=1}^{g_{ij}}U_j\otimes_{\OO_{ij}}\OO_i. \]
It is free over $\OO_i$ of rank $\sum_{j\ne i}(-c_{ij})r_j$, where $\mathbf r=\underline{\rk}U$. The incoming map $A_i:W_i(U)\to U_i$ sends $u\otimes a$ in the $(j,c)$ summand to $(ua_{ji}^{(c)})a$. The outgoing arrows determine, by the Frobenius adjunction, a map $B_i:U_i\to W_i(U)$ whose $(j,c)$ component is
\[ B_i(v)_{j,c}=\epsilon_{ij}\sum_{q=0}^{f_{ji}-1}\bigl(vx_i^{f_{ji}-1-q}a_{ij}^{(c)}\bigr)\otimes x_i^q. \]
Thus
\[ W_i(U)\xrightarrow{A_i}U_i\xrightarrow{B_i}W_i(U),\qquad (A_i\circ B_i)(v)=v\rho_i. \]
We write $A_iB_i$ for $A_i\circ B_i$. The moment relation is $A_iB_i=0$; it does not require $B_iA_i=0$. The map $B_i$ and the outgoing arrow maps determine one another. The construction is valid over the generic fibre and all central reductions, and commutes with scalar extension.

\subsection{Duality}\label{sec:duality}

The anti-involution $*$ of $\PP_d$ fixes $1_i,x_i$ and interchanges $a_{ij}^{(c)}$ with $a_{ji}^{(c)}$. Reversing the summation index preserves each moment relation. For a locally free lattice $M$ over $\Bbbk[[t]]$, with $\Bbbk\supseteq\C$ algebraically closed, put
\[ M^\vee=\Hom_{\Bbbk[[t]]}(M,\Bbbk[[t]]),\qquad (\varphi a)(m)=\varphi(ma^*). \]
The coefficient Frobenius pairings make $M^\vee$ vertexwise free of the same ranks. Duality commutes with reduction; at length $k$, composition with the coefficient of $t^{k-1}$ identifies the relative dual with the ordinary $\Bbbk$-linear dual. It is exact on lattice sequences with locally free quotients. At finite length it preserves $E$-filtrations and maximal components, inducing $Z\mapsto Z^*$.

Annihilators reverse admissible flags, over their ground field. For crystal modules at finite length they give
\[ \fac_i(M^\vee)\cong\op{sub}_i(M)^\vee,\qquad
\mathsf K_i(M^\vee)\cong(M/\op{sub}_i(M))^\vee. \]
Thus top and socle correspondences are dual. The integral splitting needed for canonical removals is established in Section~\ref{sec:perfect}.

\subsection{Splitting the vertex fields}
Put $L=\lcm_i d_i$ and $K^{(L)}=\C((t^{1/L}))$. After scalar extension to $K^{(L)}$, each $K_i$ is a product of $d_i$ copies of $K^{(L)}$, indexed by the eigenvalues $\xi$ of $x_i$, with $\xi^{d_i}=t$. The commutation relation \eqref{eq:comm} permits an arrow from $\xi$ at $i$ to $\eta$ at $j$ precisely when $\xi^{f_{ji}}=\eta^{f_{ij}}$. A fixed vertex above $i$ is incident with $g_{ij}f_{ij}=-c_{ij}$ arrows joining it to vertices above $j$.

For two distinct eigenvalues $\xi,\xi'$ above $i$, the coefficient in the $i$-moment relation is
\[ \sum_{q=0}^{a-1}\xi^q(\xi')^{a-1-q}=\frac{\xi^a-(\xi')^a}{\xi-\xi'},\qquad a=f_{ji}. \]
It vanishes on an allowed two-arrow path, since its endpoints have the same $a$th power. For equal eigenvalues, the coefficient is $a\xi^{a-1}$. Multiplying the $i$-equation by $x_i/d_i$ makes the edge coefficient $z_{ij}/e_{ij}$, where $z_{ij}=\xi^{f_{ji}}=\eta^{f_{ij}}$. The coefficient is the same at the opposite endpoint. Rescaling one arrow of each doubled pair by this nonzero coefficient gives the ordinary signed preprojective relations. We thereby identify the split algebra with the ordinary preprojective algebra of this split quiver, compatibly with Galois descent. The acyclic vertex and arrow construction is that of \cite[Section~4]{GLSRS}; its arrow bimodules are the composita of the vertex fields and agree with $K_i\otimes_{K_{ij}}K_j$ here because $f_{ij}$ and $f_{ji}$ are coprime.

\section{Constructible classes modulo induced classes}\label{sec:induced}
In this section and Sections~\ref{sec:serre}--\ref{sec:hopf}, we allow $K=\Bbbk((t))$ with $\Bbbk\supseteq\C$ algebraically closed, and keep the notation for the algebras after this coefficient extension. For any field extension $L/K$ and $K$-scheme $Y$, write $Y_L=Y\times_K\op{Spec}L$. Choose compatible roots of $t$ and set
\[ K^{(e)}=\Bbbk((t^{1/e})),\qquad K_i=K^{(d_i)}. \]
Every finite extension of $K$ in the chosen algebraic closure is one of the $K^{(e)}$. Its absolute Galois group $G_K=\op{Gal}(\overline K/K)$ is isomorphic to the profinite completion $\widehat\Z$ of $\Z$. We choose a topological generator $\sigma$; the subgroup $G_{K^{(e)}}$ corresponds to $e\widehat\Z$.

Here $\op{Ind}_{G_{K^{(e)}}}^{G_K}$ denotes induction of continuous representations along this finite-index inclusion. If $e>1$, the trace of $\sigma$ on an induced representation is zero, since it permutes the coset summands without fixing one. This motivates the quotient by properly induced classes defined below. The following example explains the calculation.

\begin{example}\label{ex:swap}
For the quadratic extension $K^{(2)}/K$, the geometric fibre of $\Res_{K^{(2)}/K}\mb{P}^1$ is $(\mb{P}^1)^2$. Its cohomology is concentrated in even degrees. We use $\Q_\ell(m)$ for the $m$th Tate twist; its Galois action is trivial here because $K$ contains all roots of unity. As an ungraded Galois representation, the total cohomology is
\[ H^*((\mb{P}^1)^2,\Q_\ell)\cong\Q_\ell^{\oplus2}\oplus\op{Ind}_{G_{K^{(2)}}}^{G_K}\Q_\ell. \]
The two classes in degree two are exchanged by inertia. The induced summand has trace zero, so the monodromy trace is $2$, whereas the geometric Euler characteristic is $4$.
\end{example}

\subsection{The quotient of Grothendieck groups}
Let $Y$ be a separated finite-type $K$-scheme. We write $D_c^b(Y,\Q_\ell)$ for the bounded derived category of $\ell$-adic complexes on $Y$ with constructible cohomology sheaves. A sheaf is constructible if it is locally constant of finite rank on each member of a finite stratification. Set
\[ \ms{K}(Y)=K_0(D_c^b(Y,\Q_\ell)). \]
This Grothendieck group is generated by classes $[\mc{A}]$, with $[\mc{B}]=[\mc{A}]+[\mc{C}]$ for a distinguished triangle $\mc{A}\to\mc{B}\to\mc{C}\to\mc{A}[1]$. Derived tensor product makes it a ring. We write $\Q_{\ell,Y}$ for the constant sheaf on $Y$, omitting $Y$ when its meaning is clear. For a morphism $f$, $f^*$ denotes pullback and $Rf_!$ denotes derived direct image with proper support. These are the constructible \'etale operations of \cite{SGA4half}.

For $e\ge1$, let $\pi_{e,Y}:Y_{K^{(e)}}\to Y$ be the finite \'etale projection; write $\pi_e=\pi_{e,\op{Spec}K}$. A \emph{properly induced class} is a class
\begin{equation}\label{eq:induced-subgroup}
 [R\pi_{e,Y!}\mc{A}],\qquad e>1,\quad\mc{A}\in D_c^b(Y_{K^{(e)}},\Q_\ell).
\end{equation}
Thus ``properly'' refers to an extension of degree greater than one. Let $\ms{N}(Y)$ be the subgroup generated by these classes, and set
\[ \overline{\ms{K}}(Y)=\ms{K}(Y)/\ms{N}(Y). \]
A subscript $\Q$ on either group means tensor product with $\Q$. All parameter-scheme morphisms below are over the fixed field $K$.

\begin{proposition}\label{prop:induced-operations}
The subgroup $\ms{N}(Y)$ is an ideal for tensor product. Pullback, direct image with proper support, and exterior products preserve induced classes. Hence they descend to the groups $\overline{\ms{K}}$.
\end{proposition}
\begin{proof}
For $\mc{B}\in D_c^b(Y,\Q_\ell)$, the projection formula gives
\[ (R\pi_{e,Y!}\mc{A})\otimes^{\mathbf L}\mc{B}\cong R\pi_{e,Y!}(\mc{A}\otimes^{\mathbf L}\pi_{e,Y}^*\mc{B}). \]
This proves the ideal assertion. If $f:X\to Y$ is separated and of finite type, let $f_e:X_{K^{(e)}}\to Y_{K^{(e)}}$ be its scalar extension. Base change for the finite map $\pi_{e,Y}$ gives
\[ f^*R\pi_{e,Y!}\cong R\pi_{e,X!}f_e^*. \]
The identity $f\pi_{e,X}=\pi_{e,Y}f_e$ also gives
\[ Rf_!R\pi_{e,X!}\cong R\pi_{e,Y!}Rf_{e!}. \]
These formulas prove preservation under pullback and direct image. For an exterior product, write $\mc{A}\boxtimes\mc{B}=p_X^*\mc{A}\otimes^{\mathbf L}p_Y^*\mc{B}$ on $X\times_KY$, where $p_X,p_Y$ are the projections, and apply the same statements.
\end{proof}

For $\mc{A}\in D_c^b(Y,\Q_\ell)$, define the trace with coefficients by
\[ \chi_\sigma(Y,\mc{A})=\sum_q(-1)^q\Tr\!\left(\sigma\mid H_c^q(Y_{\overline K},\mc{A}|_{Y_{\overline K}})\right). \]
Constant coefficients recover $\chi_\sigma(Y)$.
\begin{lemma}\label{lem:induced-trace}
The functional $\chi_\sigma(Y,-)$ factors through $\overline{\ms{K}}(Y)$.
\end{lemma}
\begin{proof}
The compactly supported cohomology of $R\pi_{e,Y!}\mc{A}$ is induced from $G_{K^{(e)}}$. For $e>1$, the generator $\sigma$ moves every coset summand to a different summand, so its trace is zero in every degree. This is the induced-representation form of \cite[Lemma~7.3]{FeiQP}.
\end{proof}

\begin{remark}\label{rem:not-category}
The quotient is a quotient of groups, not of triangulated categories. Indeed, $R\pi_{2!}\Q_\ell$ contains $\Q_\ell$ as a direct summand. A thick tensor subcategory containing the induced objects would therefore contain the tensor unit. In our quotient, $[\Q_\ell]\ne0$ because its trace is one.
\end{remark}

\subsection{Cyclic permutations and traces}
We begin with an elementary calculation for a graded vector space concentrated in even degrees. Its dimension means the sum of the dimensions of its graded parts.
\begin{lemma}\label{lem:necklace}
Let $V$ be a finite-dimensional graded $\Q_\ell$-vector space concentrated in even degrees, with trivial Galois action. For an integer $n\ge1$, let $\sigma$ act on $V^{\otimes n}$ by cyclic permutation of the factors. Then
\[ [V^{\otimes n}]\equiv(\dim V)[\Q_\ell]\pmod{\ms{N}(\op{Spec}K)}. \]
\end{lemma}
\begin{proof}
Choose a homogeneous basis. The action permutes the tensor basis without signs. A basis tensor is fixed by the cyclic generator only when all its factors are the same basis vector, giving $\dim V$ fixed tensors. Every other orbit has length $e>1$ dividing $n$. Its span is the permutation representation of $G_K/G_{K^{(e)}}$, induced from $G_{K^{(e)}}$.
\end{proof}

For the flag varieties used below, cohomology is concentrated in even degrees and each graded part is a direct sum of Tate twists. Since these twists have trivial Galois action over $K$, their alternating cohomology class is the class of their total cohomology. This allows us to apply Lemma~\ref{lem:necklace}.

We also need the graded trace identity when odd cohomology is present. A \emph{super vector space} is a $\Z/2\Z$-graded vector space $V=V_{\overline0}\oplus V_{\overline1}$. A homogeneous vector $v$ has parity $|v|\in\{0,1\}$. An \emph{even endomorphism} preserves these two summands, and its \emph{supertrace} is $\op{str}(T)=\Tr(T|_{V_{\overline0}})-\Tr(T|_{V_{\overline1}})$.
\begin{lemma}[Cyclic supertrace]\label{lem:cyclic}
Let $V$ be a finite-dimensional super vector space, let $T$ be an even endomorphism, and let $n\ge1$. Define on homogeneous tensors
\[ C_T(v_1\otimes\cdots\otimes v_n)=(-1)^{|v_n|(|v_1|+\cdots+|v_{n-1}|)}T(v_n)\otimes v_1\otimes\cdots\otimes v_{n-1}. \]
Then $\op{str}(C_T\mid V^{\otimes n})=\op{str}(T\mid V)$. In particular, if $t=u^n$ and $X$ is a quasi-projective $\Bbbk((u))$-scheme, then
\begin{equation}\label{eq:weiltrace}
 \chi_{\sigma_t}\!\left(\Res_{\Bbbk((u))/\Bbbk((t))}X\right)=\chi_{\sigma_u}(X),\qquad\sigma_u=\sigma_t^n,
\end{equation}
where $\sigma_t$ and $\sigma_u$ are compatible tame generators over $\Bbbk((t))$ and $\Bbbk((u))$, respectively.
\end{lemma}
\begin{proof}
A diagonal entry of $C_T$ in a homogeneous tensor basis uses the same basis vector in every factor. If that vector has parity $p$, the permutation contributes $(-1)^{(n-1)p}$ and the supertrace contributes $(-1)^{np}$. Their product is $(-1)^p$, proving the first assertion. Over an algebraic closure, the Weil restriction is the product of the $n$ conjugates of $X$. The K\"unneth formula gives the cyclic action on their cohomology, and its $n$th power on one factor is $\sigma_u$. Apply the first assertion. This extends the Grassmannian calculation of \cite[Lemma~7.4]{FeiQP}.
\end{proof}

\subsection{Relative flag cohomology}
We extend the relative Grassmannian calculation of \cite[Lemma~7.4]{FeiQP} to the two flag refinements needed for the Serre relation. Fix integers $m,a,b\ge0$ with $a+b\le m$, an $m$-dimensional complex vector space $V_0$, and subspaces $V_-\subseteq V_+\subseteq V_0$ of dimensions $a$ and $m-b$. For $0\le r\le m$, let $P_r(a,b;m)$ parametrize
$V_-\subseteq U\subseteq V_+$ with $\dim U=r$,
together with complete flags in $U$ and $V_0/U$. It is empty unless $a\le r\le m-b$. Otherwise, forgetting the two flags makes it a pair of flag bundles over $\Gr_{r-a}(V_+/V_-)$. Its cohomology is concentrated in even degrees and is a sum of Tate twists, and
\[ h_r:=\chi_c(P_r(a,b;m))=r!(m-r)!\binom{m-a-b}{r-a}. \]
We set a binomial coefficient equal to zero when its lower index is outside the usual range.

\begin{proposition}[Relative cyclic reduction]\label{lem:relative}
Let $Y$ be a separated finite-type $K$-scheme. On $Y_{K_i}$, let $\mc{A}\subset\mc{B}\subset\mc{V}$ be subbundles of ranks $a,m-b,m$, with locally free quotients. Let $\mc{P}_r\to Y_{K_i}$ be the associated relative flag scheme of the preceding paragraph, and let
\[ p_r:\mc{F}_r=\Res_{Y_{K_i}/Y}\mc{P}_r\longrightarrow Y. \]
Thus an $S$-point of $\mc{F}_r$, for a $Y$-scheme $S$, is such a flag over $S\times_Y Y_{K_i}$. Then
\begin{equation}\label{eq:relative-class}
 [Rp_{r!}\Q_\ell]=h_r[\Q_{\ell,Y}]\quad\text{in }\overline{\ms{K}}(Y).
\end{equation}
In particular, for every $\mc{L}\in D_c^b(Y,\Q_\ell)$,
\begin{equation}\label{eq:relative}
 \chi_\sigma(\mc{F}_r,p_r^*\mc{L})=h_r\chi_\sigma(Y,\mc{L}).
\end{equation}
\end{proposition}
\begin{proof}
Over a finite splitting extension, $\mc{F}_r$ is the product over $Y$ of the flag families indexed by the embeddings of $K_i$. For each embedding, the family is a Grassmann bundle for $\mc{B}/\mc{A}$, followed by the complete-flag bundles of its universal subbundle $U$ and the quotient $\mc{V}/U$.

The Schur polynomials in the Chern classes of the tautological Grassmann bundle, indexed by partitions in the appropriate rectangle, give a basis on each geometric fibre. Products with the usual Chern monomials of the two flag bundles give a basis of the entire fibre cohomology. These classes are defined on the total family. Proper base change identifies the resulting maps from constant Tate sheaves to the higher direct images as isomorphisms on every geometric stalk, hence as isomorphisms of sheaves. This is the relative Leray--Hirsch argument used in \cite[Lemma~7.4]{FeiQP}; smoothness of $Y$ is not required.

Put $V_r=\bigoplus_qH^q(P_r(a,b;m),\Q_\ell)$. The classes are natural under isomorphisms of the three bundles. Consequently, Galois descent permutes the factors indexed by embeddings, and $\sigma^{d_i}$ fixes the classes in each factor. The K\"unneth formula identifies the arithmetic higher direct images with the graded parts of $V_r^{\otimes d_i}$ pulled back from $K$. Their degrees are even, and the Tate twists are trivial over $K$. If $a_Y:Y\to\op{Spec}K$ is the structure map, then
\[ [Rp_{r!}\Q_\ell]=a_Y^*[V_r^{\otimes d_i}]\quad\text{in }\ms{K}(Y). \]
Lemma~\ref{lem:necklace} and Proposition~\ref{prop:induced-operations} give \eqref{eq:relative-class}. Tensoring with $\mc{L}$, applying the projection formula, and taking compactly supported cohomology give \eqref{eq:relative}. All constructions commute with base change. Nilpotent structure on the base does not affect the \'etale cohomology or these identities.
\end{proof}

\section{The relative Serre relation}\label{sec:serre}
We keep the field $K=\Bbbk((t))$ and the quotient $\overline{\ms{K}}(Y)$ of Section~\ref{sec:induced}. Fix distinct vertices $i,j$ and put $m=1-c_{ij}$. Let $Y$ be a separated finite-type $K$-scheme, and let $\mc{M}$ be a family of right $\Gamma_d$-modules on $Y$. Its vertex bundles have rank $m$ at $i$, rank one at $j$, and rank zero elsewhere, over the corresponding vertex algebras.

For $0\le r\le m$, let
$p_r:\mc{F}_r\rightarrow Y$
be the scheme of balanced flags in $\mc{M}$ with successive quotients prescribed by the word $i^rji^{m-r}$. Here $i^r$ means $r$ consecutive occurrences of $i$. With our ascending flag convention,
\[ \sum_{r=0}^m(-1)^r\binom mr z_i^r z_j z_i^{m-r}=(-1)^m(\op{ad}z_i)^m z_j. \]
Thus the alternating sum below represents a Serre relation, up to the constant sign $(-1)^m$.

\begin{theorem}[Serre relation modulo induced classes]\label{thm:Serre-class}
For the family above,
\begin{equation}\label{eq:Serre-class}
 \sum_{r=0}^{m}(-1)^r\binom mr[Rp_{r!}\Q_\ell]=0
 \quad\text{in }\overline{\ms{K}}(Y).
\end{equation}
The equality remains valid after pullback along a morphism of $K$-schemes, tensoring with a constructible complex on $Y$, or direct image with proper support.
\end{theorem}
\begin{proof}
We may replace $Y$ by its reduction, since this does not change the constructible \^{e}tale categories. After splitting the vertex fields, the incoming and outgoing maps at each vertex above $i$ have the form
\[ W\xrightarrow{A}V\xrightarrow{B}W,\qquad AB=0,\qquad \dim V=m,\quad\dim W=-c_{ij}=m-1. \]
These dimensions are the ranks of the respective vector bundles. At a geometric point, put $a=\rk A$ and $b=\rk B$. Since $\im B\subseteq\ker A$ in $W$, we have $a+b\le m-1$.

Let $U$ be the $i$-subspace just before the quotient at $j$ is inserted. Before that insertion, arrow stability requires $U\subseteq\ker B$; after it, stability requires $\im A\subseteq U$. Thus
$\im A\subseteq U\subseteq\ker B$ and $\dim U=r$.
Such a subspace can exist only when $BA=0$. On the locus where $BA\ne0$ at some vertex above $i$, every $\mc{F}_r$ is empty. We therefore restrict to the complementary locus. Once $U$ is chosen, the remaining choices are complete flags in $U$ and $V/U$; the one-dimensional space at $j$ has no flag choices.

We stratify this locus by the ranks of $A$ and $B$ at all vertices above $i$. Label the embeddings of $K_i$ cyclically by $\nu\in\Z/d_i\Z$, so that the tame generator acts by a cyclic shift, and write
\[ \boldsymbol\tau=((a_\nu,b_\nu))_{\nu\in\Z/d_i\Z} \]
for a tuple of ranks. This gives a finite locally closed stratification. Indeed, first fix the total $\sum_\nu(a_\nu+b_\nu)$, which is the rank of a direct sum of bundle maps. On this locus, imposing the corresponding lower bounds on all the individual ranks is an open condition and forces equality. Each tuple stratum is therefore open and closed in the locus with fixed total rank. The unions over cyclic orbits are locally closed and defined over $K$.

If the cyclic orbit of $\boldsymbol\tau$ has length $e>1$, its stabilizer is $G_{K^{(e)}}$. The individual stratum and its restricted flag families descend to $K^{(e)}$, since their rank conditions are preserved by this subgroup. Extend the resulting constructible class by zero to $Y_{K^{(e)}}$ and push it forward by $\pi_{e,Y}$. After splitting, this pushforward is the sum over the $e$ distinct conjugate strata. It is exactly the flag contribution over the orbit union, and hence belongs to $\ms{N}(Y)$ for each $r$.

It remains to consider constant tuples, with $a_\nu=a$ and $b_\nu=b$ for every $\nu$. Denote the corresponding stratum by $Y_{a,b}$. On $(Y_{a,b})_{K_i}$, the bundles $\im A_i$, $\ker B_i$, and $\mc{M}_i$ have ranks $a,m-b,m$. Their quotients are locally free: locally an appropriate minor is invertible, while the larger minors vanish on the reduced rank stratum. The restricted flag family is therefore the family in Proposition~\ref{lem:relative}, with $\mc{A}=\im A_i$ and $\mc{B}=\ker B_i$.

Put $h=m-a-b\ge1$. Proposition~\ref{lem:relative} gives the following contribution on $Y_{a,b}$:
\[
\begin{aligned}
 &\sum_{r=0}^m(-1)^r\binom mr r!(m-r)!\binom h{r-a}\,[\Q_{\ell,Y_{a,b}}]\\
 =&m!(-1)^a\sum_{q=0}^h(-1)^q\binom hq\,[\Q_{\ell,Y_{a,b}}]=0.
\end{aligned}
\]
We view a class on a locally closed stratum as a class on $Y$ by extension by zero. Additivity for a finite stratification then proves \eqref{eq:Serre-class}. The compatibility with the three operations follows from Proposition~\ref{prop:induced-operations}.
\end{proof}

\begin{corollary}[Trace Serre relations]\label{thm:Serre}
For the family above and every $\mc{L}\in D_c^b(Y,\Q_\ell)$,
\[ \sum_{r=0}^{m}(-1)^r\binom mr\chi_\sigma(\mc{F}_r,p_r^*\mc{L})=0. \]
For every nilpotent $\Gamma_d$-module $M$ and every pair of words $\mathbf u,\mathbf v$,
\[ \sum_{r=0}^{m}(-1)^r\binom mr F_{\mathbf u\,i^rji^{m-r}\,\mathbf v}(M)=0. \]
Consequently, the linear evaluation map $\ev_\sigma:T_\Q\to\mc{M}_\sigma(\Gamma_d)$ annihilates $J_C$.
\end{corollary}
\begin{proof}
Tensor \eqref{eq:Serre-class} with $[\mc{L}]$, take direct image with proper support to $\op{Spec}K$, and apply Lemma~\ref{lem:induced-trace}. The projection formula gives the first identity.

For the second identity, the expression vanishes unless the word has the rank of $M$. In the matching degree, use the partial-flag scheme in which the quotients prescribed by $\mathbf u$ and $\mathbf v$ remain refined, while the intervening subquotient is left unrefined. The latter is a family of rank $m\alpha_i+\alpha_j$. Its relative flags are precisely the full flags in the second identity, so the first identity with constant coefficients applies. As $i,j,\mathbf u,\mathbf v$ vary, these expressions span the two-sided Serre ideal.
\end{proof}

The proof also respects changes of vertex bases: the rank conditions and the tautological cohomology classes are preserved. We use this equivariance in the construction of convolution below.

\section{Geometric composition and the Hopf structure}\label{sec:hopf}
We construct convolution from subrepresentations, using Lusztig's induction correspondence \cite{Lusztig1991}. We continue to work over the field $K=\Bbbk((t))$ of Section~\ref{sec:induced}.

\subsection{Representation spaces and equivariant classes}
For $\mathbf r\in\mb{N}^I$, set $n(\mathbf r)=\sum_i d_ir_i$, the dimension over $K$ of the vertex space $\bigoplus_iK_i^{r_i}$. Let $\mc{R}_{\mathbf r}$ be the affine $K$-scheme of $\Gamma_d$-representations on these fixed vertex spaces satisfying $MJ^{n(\mathbf r)}=0$.
The arrow maps, the defining algebra relations, and this nilpotence condition give finitely many polynomial equations. We put $\mc{R}_0=\op{Spec}K$.

The bound $MJ^{n(\mathbf r)}=0$ holds for every geometric nilpotent representation, since its simple factors are one-dimensional after the vertex fields split. A nilpotent module defined over $K$ has the smaller composition length $|\mathbf r|=\sum_i r_i$, since its simple factors are the $S_i$. Moreover, a balanced full flag forces $MJ^{|\mathbf r|}=0$ even in families, because $J$ sends each member of the flag into its predecessor. Thus the full-flag complexes are supported on the closed subscheme $MJ^{|\mathbf r|}=0$. Extension by zero identifies their classes whether this subscheme or $MJ^{n(\mathbf r)}=0$ is used as the ambient representation space.

The group $G_{\mathbf r}=\prod_i\Res_{K_i/K}\op{GL}_{r_i}$
acts on $\mc{R}_{\mathbf r}$ by changes of vertex bases. More generally, if an algebraic group $G$ acts on a $K$-scheme $X$, let $\ms{K}_G(X)$ be the Grothendieck group of constructible $\Q_\ell$-sheaves with a $G$-\emph{equivariant structure}. Such a structure is an isomorphism between the pullbacks by the action and projection maps $G\times X\to X$, satisfying the cocycle condition. Derived operations on classes are taken as alternating sums of their higher direct images.

Define $\ms{N}_G(X)$ by the finite pushforwards in \eqref{eq:induced-subgroup}, using sheaves on $X_{K^{(e)}}$ equivariant for the algebraic group $G\times_K\op{Spec}K^{(e)}$. We use the rational quotient
\[ \overline{\ms{K}}_G(X)_\Q=\bigl(\ms{K}_G(X)/\ms{N}_G(X)\bigr)\otimes_\Z\Q. \]

\subsection{Induction and convolution}
Fix $\mathbf a,\mathbf b\in\mb{N}^I$ and put $\mathbf r=\mathbf a+\mathbf b$. Let $\mc{E}_{\mathbf a,\mathbf b}$ parametrize pairs $(M,U)$, where $M\in\mc{R}_{\mathbf r}$ and $U\subseteq M$ is a subrepresentation of balanced rank $\mathbf a$. We require
\[ UJ^{n(\mathbf a)}=0,\qquad (M/U)J^{n(\mathbf b)}=0. \]
Choose vertex bases in $U$ and $M/U$ to obtain the scheme $\widetilde{\mc{E}}_{\mathbf a,\mathbf b}$ and the maps
\[ \mc{R}_{\mathbf a}\times\mc{R}_{\mathbf b}\xleftarrow{\ p_1\ }\widetilde{\mc{E}}_{\mathbf a,\mathbf b}\xrightarrow{\ p_2\ }\mc{E}_{\mathbf a,\mathbf b}\xrightarrow{\ p_3\ }\mc{R}_{\mathbf r}. \]
Here $p_1$ gives the two representations in the chosen bases, $p_2$ forgets those bases, and $p_3$ forgets $U$. The map $p_2$ is a principal $G_{\mathbf a}\times G_{\mathbf b}$-bundle: its fibres are the pairs of choices of vertex bases. The map $p_3$ is proper, because the balanced vertex subspaces form a product of Weil restrictions of Grassmannians, and arrow stability and the two nilpotence conditions are closed conditions.

Let $\mc{A}$ and $\mc{B}$ be equivariant constructible sheaves on $\mc{R}_{\mathbf a}$ and $\mc{R}_{\mathbf b}$. The sheaf $p_1^*(\mc{A}\boxtimes\mc{B})$ descends through $p_2$ to a sheaf $\mc{H}$ on $\mc{E}_{\mathbf a,\mathbf b}$. We define their \emph{convolution} by
\[ [\mc{A}]*[\mc{B}]=[Rp_{3!}\mc{H}]. \]
The first factor is evaluated on the submodule, so this convention agrees with the ascending flags used throughout. We insert no cohomological shifts.

Associativity follows from the common three-step flag correspondence of \cite{Lusztig1991}. If a submodule is annihilated by $J^a$ and its quotient by $J^b$, the whole module is annihilated by $J^{a+b}$. Since $n(\mathbf a+\mathbf b)=n(\mathbf a)+n(\mathbf b)$, both iterated products use the same flags $0\subset M^1\subset M^2\subset M$ with the prescribed successive quotients. Proper base change identifies their pushforwards. The unit is the constant class at the zero representation.

\begin{proposition}\label{prop:Groth-convolution}
The convolution descends to
$\bigoplus_{\mathbf r}\overline{\ms{K}}_{G_{\mathbf r}}(\mc{R}_{\mathbf r})_\Q$.
The \emph{composition algebra} $\ms{C}_d$ is the unital subalgebra generated by the constant classes at the simple representations $S_i$. Sending $e_i$ to the class at $S_i$ gives a surjective algebra homomorphism
$U^+_\Q\longrightarrow\ms{C}_d$.
\end{proposition}
\begin{proof}
Exterior products, pullback, and proper pushforward preserve induced classes by Proposition~\ref{prop:induced-operations}. Descent through $p_2$ preserves them as well. To see this, write $q=p_2$ and let $q_e$ be its scalar extension to $K^{(e)}$. The Cartesian square gives
$q^*R\pi_{e,\mc{E}_{\mathbf a,\mathbf b}!}\cong R\pi_{e,\widetilde{\mc{E}}_{\mathbf a,\mathbf b}!}q_e^*$.
The equivalence between sheaves on the base of a principal bundle and equivariant sheaves on its total space therefore identifies the descent of an induced class with an induced class. The action of $G_{\mathbf r}$ is preserved throughout. This proves that convolution descends to the quotient.

The scheme $\mc{R}_{\alpha_i}$ is the point representing $S_i$. An iterated convolution of these constant classes is the proper pushforward from the universal balanced full-flag scheme. The rank strata and tautological classes in Theorem~\ref{thm:Serre-class} are invariant under vertex changes of bases, so its proof gives the same relation in the equivariant Grothendieck group. Thus the simple classes satisfy the Serre relations. The homomorphism from $T_\Q$ factors through $U^+_\Q=T_\Q/J_C$, and it is surjective onto $\ms{C}_d$ by definition.
\end{proof}

For a $K$-rational representation $M$, restriction to $M$ followed by tame trace defines a linear evaluation on $\ms{C}_d$. Proper base change identifies the value of the class of a word $w$ with $F_w(M)$. For $K=\C((t))$, Section~\ref{sec:basis} proves that the composite
$U^+_\Q\longrightarrow\ms{C}_d\longrightarrow\mc{M}_\sigma(\Gamma_d)$
is injective, by triangular evaluations on representations over this field. Both maps are surjective onto their indicated targets. They are therefore isomorphisms, and the geometric convolution induces multiplication of the trace functions.

\subsection{Direct sums}
We next compute the coproduct. If $w=(i_1,\ldots,i_h)$ and $S\subseteq\{1,\ldots,h\}$, let $w|_S$ be the subword with positions in $S$, in their original order, and let $S^c$ be its complementary set of positions.

\begin{proposition}[Geometric coproduct]\label{prop:coproduct}
For nilpotent $\Gamma_d$-modules $M,N$ and any word $w$,
\begin{equation}\label{eq:coproduct}
 F_w(M\oplus N)=\sum_{S\subseteq\{1,\ldots,h\}}F_{w|_S}(M)F_{w|_{S^c}}(N).
\end{equation}
\end{proposition}
\begin{proof}
We use the argument of \cite[Proposition~7.6]{FeiQP} for flags; compare \cite[Section~6]{GLS2005}. Given a flag $(L_a)_{a=0}^h$ in $M\oplus N$, project it to $M$ and intersect it with $N$:
$U_a=\op{pr}_M(L_a)$ and $V_a=L_a\cap N$.
After splitting the vertex fields, stratify by the dimensions of these submodules. Fix the resulting flags $(U_a)$ and $(V_a)$, allowing repeated members. A submodule $L_a$ giving them is uniquely determined by a map
\[ \phi_a:U_a\longrightarrow N/V_a,\qquad L_a=\{(u,n):u\in U_a,\ n+V_a=\phi_a(u)\}. \]
It is arrow-stable exactly when $\phi_a$ is a module homomorphism. For $a\le b$, let $q_{a,b}:N/V_a\to N/V_b$ be the quotient map. The inclusion $L_a\subseteq L_b$ is equivalent to $q_{a,b}\phi_a=\phi_b|_{U_a}$. These are homogeneous linear equations. Their solution spaces form vector bundles after stratifying further by the ranks of the equations. The zero maps give a section. An affine fibre of dimension $v$ has compactly supported cohomology $\Q_\ell(-v)$ in degree $2v$ and zero in other degrees, so it contributes trace one.

At each position of $w$, the one-dimensional quotient at a split vertex is assigned either to the projected flag in $M$ or to the intersection flag in $N$. Consider the entire tuple of assignments at all word positions. If an assignment differs among the sheets above some original vertex, inertia moves it in a nontrivial cyclic orbit. Its contribution is induced from a proper finite extension and has trace zero, by \cite[Lemma~7.3]{FeiQP}. A constant tuple assigns each letter wholly to one summand and hence determines a subset $S$ of positions. The corresponding base is the product of the balanced flag schemes for $w|_S$ in $M$ and $w|_{S^c}$ in $N$. The K\"unneth formula and the affine-fibre calculation prove \eqref{eq:coproduct}.
\end{proof}

The relative complete-flag variety refining a successive quotient $S_i^p$ has monodromy trace $p!$, by Proposition~\ref{lem:relative}. Applying \eqref{eq:coproduct} to complete refinements therefore also gives the partial-flag formula. For each decomposition $p=a+b$, the choice of the $a$ letters assigned to $M$ contributes $\binom pa$, and the two refinement factors are $a!$ and $b!$. Their product is $p!$, which cancels the normalization by the complete refinement. Thus the coproduct $\Delta$ on divided powers is
\[ \Delta(e_i^{(p)})=\sum_{a+b=p}e_i^{(a)}\otimes e_i^{(b)}. \]

Once faithfulness is proved, \eqref{eq:coproduct} identifies $\Delta$ with the standard coproduct of $U^+_\Q$. The counit is evaluation at the zero module. The degree-zero part is $\Q$, so the graded bialgebra has an antipode $\mathsf S$, given on words by
\[ \mathsf S(F_{i_1\cdots i_h})=(-1)^hF_{i_h\cdots i_1}. \]
The image of $U^+_\Z$ is consequently an integral Hopf algebra.

\section{Integral charts and reduction of components}\label{sec:charts}
We return to $\OO=\C[[t]]$ and compare the maximal components of $\Pi_k=\PP_d/t^k\PP_d$. Their vertex modules $H_{i,k}^{r_i}$ and power bases are fixed. We first lift extensions by a full $i$-top, then construct compatible affine charts, and finally identify the component sets at different lengths. The special-fibre result used is the maximal-top correspondence recalled in Section~\ref{sec:setup}.

\subsection{Extensions by a full vertex-supported quotient}
\begin{lemma}[Free maximal-top extension]\label{lem:extension}
Let $U$ be a locally free $\PP_d$-lattice, fix $i\in I$ and $p\ge0$, and assume that $A_i(U)$ is surjective. Consider extensions of $U$ by the module supported at $i$ with vertex space $\OO_i^p$. Fix the inclusion of $U$, the quotient identification, and the splitting $M_i=U_i\oplus\OO_i^p$. The possible arrow actions are parametrized by
\[ \gamma\in\Hom_{\OO_i}(\OO_i^p,\ker A_i(U)),\qquad A_i(M)=\binom{A_i(U)}{0},\quad B_i(M)=(B_i(U)\ \gamma). \]
For every such extension, $\fac_i(M)\cong\OO_i^p$. The same description holds modulo $t^k$ and in families on any open subset where a fixed maximal minor of $A_i(U)$ is invertible.
\end{lemma}
\begin{proof}
All vertex modules except the one at $i$ are unchanged, so $W_i(M)=W_i(U)$. The inclusion of $U$ fixes the old outgoing maps, and there is no incoming component in the new quotient. The moment relation at $i$ is therefore $A_i(U)\gamma=0$. At another vertex, a path passing through $i$ enters the old summand $U_i$, so its moment relation is unchanged. The Frobenius adjunction of Section~\ref{sec:adjunction} reconstructs the outgoing arrows from $\gamma$. Since the incoming image is precisely $U_i$, the full $i$-top is the prescribed quotient.

A surjection onto a free module splits. On the open set where a chosen maximal minor is invertible, its columns allow us to express the corresponding coordinates in terms of the remaining ones and thereby give a basis of the kernel. The \emph{norm} of its determinant means the determinant of multiplication by that element on the finite free vertex algebra over the scalar base. This norm is invertible exactly when the element is invertible. It therefore defines the required open subset over the scalar base. All the formulas commute with scalar extension, since the kernel is defined by a split exact sequence there.
\end{proof}

\begin{corollary}\label{cor:crystal-lift}
Every crystal $\Pi_1$-module has a locally free $\PP_d$-lattice reducing to it. The lift can be chosen along any prescribed complete sequence of free maximal-top removals, with the successive top ranks unchanged on the generic fibre.
\end{corollary}
\begin{proof}
We induct on total rank. A nonzero $E$-filtered module has a quotient isomorphic to some $E_i$. For a crystal module its entire $i$-top is free and nonzero. Let $U=\mathsf K_i(M)$ be the kernel of that top. It is crystal of smaller rank, and $A_i(U)$ is onto: the other vertices are unchanged and $U_i$ is the incoming image in $M_i$.

Lift $U$ by induction. The incoming map of the lift is surjective modulo $t$, hence surjective over $\OO_i$ by Nakayama's lemma. Its kernel is free and commutes with reduction. Choose a splitting at vertex $i$ for the special extension. Its parameter $\gamma_0$ then lifts to a map into the integral kernel, and Lemma~\ref{lem:extension} reconstructs a lift of $M$. Its full top is exactly the chosen free quotient. Repeating along the prescribed sequence proves the assertion about ranks. The resulting integral filtration refines to quotients supported at one vertex and free of rank one there, so its generic fibre is nilpotent.
\end{proof}

\subsection{Compatible affine charts}
\begin{proposition}[Dense compatible charts]\label{prop:charts}
For $Z\in\Bcal_1(\mathbf r)$, put $n=N_d(\mathbf r)$. There is a \emph{principal open subscheme} $\ms{X}_Z=D(f_Z)\subset\A_\OO^n$, defined by inverting a polynomial $f_Z$ whose reduction modulo $t$ is nonzero, with a universal locally free $\PP_d$-module and an $E$-filtration. Its special fibre maps isomorphically onto a dense locally closed subset of $Z$.

For $k\ge1$, let $X_{Z,k}$ be the complex scheme representing the functor
\[ X_{Z,k}(A)=\ms{X}_Z(A[t]/(t^k))\quad\text{for a commutative }\C\text{-algebra }A. \]
The right side uses the $\OO$-algebra structure sending $t$ to its residue class. In particular, $X_{Z,1}$ is the special fibre of $\ms{X}_Z$. At every length, the universal matrices identify $X_{Z,k}$ with a locally closed subset of the $E$-filtered locus for $\Pi_k$, and
\begin{equation}\label{eq:jets}
 X_{Z,k}\cong X_{Z,1}\times\A^{(k-1)n}_\C.
\end{equation}
The chart maps and their inverses commute with reduction to every smaller length.
\end{proposition}
\begin{proof}
The rank-zero case is $\ms{X}_0=\op{Spec}\OO$. For positive rank, choose a nonzero generic full top at a vertex $i$, of rank $p$, and let $Z'$ be the component obtained by its canonical removal. Its generic $i$-top is zero. We put
\[ q=r_i-p,\qquad w=\sum_{j\ne i}(-c_{ij})r_j. \]
Choose the chart for $Z'$ by induction and restrict it so that a maximal minor of its incoming map is invertible. Embed the old $i$-vertex as the graph of a map $\OO_i^q\to\OO_i^p$. In these coordinates, Lemma~\ref{lem:extension} supplies the extension parameters $\gamma$. The graph has $d_ipq$ scalar coordinates, and $\gamma$ has $d_ip(w-q)$, since the incoming kernel has rank $w-q$. Their total is
\[ d_ipq+d_ip(w-q)=d_ipw=N_d(\mathbf r)-N_d(\mathbf r-p\alpha_i). \]
The GLS maximal-top correspondence shows that the image modulo $t$ is dense in $Z$.

To recover the chart parameters from a representation in its image, order the new vertex coordinates so that the graph projects onto the first $q$. Write the incoming matrix as
\[ A_M=\begin{pmatrix}A_0\\A_1\end{pmatrix},\qquad (A_0)_{\mc{J}}\in\op{GL}_q(\OO_i), \]
where $A_0$ has $q$ rows, $A_1$ has $p$ rows, and $\mc{J}$ is a fixed set of $q$ columns. The graph matrix is
\[ X=(A_1)_{\mc{J}}\bigl((A_0)_{\mc{J}}\bigr)^{-1},\qquad A_1=XA_0. \]
Let $B_M$ be the outgoing matrix and set $g_X=\left(\begin{smallmatrix}I_q&0\\X&I_p\end{smallmatrix}\right)$. Then
\[ g_X^{-1}A_M=\binom{A_0}{0},\qquad B_Mg_X=(B_i(U)\ \gamma). \]
The first $q$ vertex coordinates and the unchanged vertices recover $U$; the last $p$ outgoing columns recover $\gamma$. The Frobenius adjunction recovers the individual arrows. These formulas give mutually inverse regular maps between the parameter chart and its image, which is defined by the graph equations and the equations of the smaller chart. Thus the image is locally closed. When $q=0$, the empty determinant is one and the same formulas apply.

As in Lemma~\ref{lem:extension}, invert the scalar norms of the finitely many required determinants. Their product, together with the earlier denominators, defines $f_Z$. All these functions are nonzero at the generic special-fibre point, so $f_Z\bmod t\ne0$. The dimension increments sum to $n$.

To describe $X_{Z,k}$, write each of the $n$ scalar coordinates as $z_{a,0}+z_{a,1}t+\cdots+z_{a,k-1}t^{k-1}$. The value of $f_Z$ is invertible precisely when $f_Z(0,z_{1,0},\ldots,z_{n,0})$ is invertible. The constant coefficients therefore belong to $X_{Z,1}$, and the remaining coefficients are unrestricted. This proves \eqref{eq:jets} and representability of the stated functor. Reduction forgets the final coefficients. The inverse graph formulas commute with this operation, and the universal $E$-filtration puts every image in the required locus.
\end{proof}

\subsection{Reduction of components}
The following elementary consequence of Smith normal form controls the supported tops and socles under reduction.
\begin{lemma}\label{lem:split-vertex}
Let $\Bbbk$ be a field, let $R_0=\Bbbk[x]/(x^n)$ with $n\ge1$, and let $A:R_0^a\to R_0^b$ be a matrix. If either $\ker A$ or $\coker A$ is free, then bases of the source and target identify $A$ with a rectangular matrix $\op{diag}(I_q,0)$. Its kernel, image, and cokernel are free, and their formation commutes with reduction to $\Bbbk[x]/(x^{n'})$ for every $1\le n'\le n$.
\end{lemma}
\begin{proof}
In Smith normal form, a nonzero nonunit factor $x^v$, with $0<v<n$, contributes $x^{n-v}R_0$ to the kernel and $R_0/(x^v)$ to the cokernel. Neither is free. Freeness of either the kernel or the cokernel therefore excludes such factors. Only unit and zero factors remain, giving the asserted matrix and split exact sequences. Reduction preserves these sequences.
\end{proof}

\begin{theorem}\label{thm:tower}
Central reduction gives compatible bijections
\[ \Bcal_k(\mathbf r)\xrightarrow{\sim}\Bcal_1(\mathbf r). \]
If $Z_k$ corresponds to $Z$, reduction $Z_k\dashrightarrow Z_l$ is dominant for every $1\le l\le k$. Each $Z_k$ contains $X_{Z,k}$ as a dense locally closed subset. The identifications commute with the top and socle crystal operators.
\end{theorem}
\begin{proof}
By \eqref{eq:jets}, the closure of $X_{Z,k}$ has dimension $kN_d(\mathbf r)$. The bound \eqref{eq:published} makes it a maximal component, denoted $Z_k$, and its reduction is dominant onto $Z$. Reduction preserves $E$-filtrations: their successive quotients are free over $\C[t]/(t^k)$, so tensoring each filtration sequence with a shorter quotient preserves exactness.

Suppose two components $Z,\widetilde Z\in\Bcal_1(\mathbf r)$ gave the same component at length $k$. The closure of the reduction of its dense $E$-filtered part would be irreducible, lie in the closure of the length-one $E$-filtered locus, and contain both $Z$ and $\widetilde Z$. Its dimension is at most $N_d(\mathbf r)$, so these two maximal components would be equal. The construction is therefore injective on component labels. The two finite sets have the same cardinality by \eqref{eq:published}, so it is bijective.

An alternative chart above $Z$, using any maximal-top removal sequence, also has a maximal component as its closure. The bijection just proved labels that component by its reduction, so its label must be $Z$. This proves independence of the choices. For $l\le k$, the coefficient projection $X_{Z,k}\to X_{Z,l}$ is dominant by \eqref{eq:jets}. Its image therefore has the unique component closure labelled by $Z$, proving compatibility of all reductions.

It remains to compare the crystal operators. Since reduction is dominant, the generic point of $Z_k$ maps to the generic point of $Z_l$. Both generic representations are crystal. Their incoming maps have free cokernels and their outgoing maps have free kernels. Lemma~\ref{lem:split-vertex} shows that the supported tops and socles, and the corresponding canonical removals, commute with reduction and have the same free ranks.

We must also identify the components of the removed modules. Choose a chart beginning with the given maximal-top removal. Its projection to the smaller chart is dominant: the remaining parameters are the graph and extension coordinates. Thus the removed module is generic in the smaller component at both lengths, and uniqueness of the reduction label identifies those components. The socle statement follows from the finite-level duality of Section~\ref{sec:duality}. A crystal string is determined by the component with zero supported top or socle and the number of successive extensions. These data are preserved, so both crystal operators commute with reduction.
\end{proof}

\section{Generic monodromy evaluations and the integral basis}\label{sec:basis}
Throughout this section, fix a rank vector $\mathbf r$ and a component $Z\in\Bcal_1(\mathbf r)$. For each $k\ge1$, let $Z_k\in\Bcal_k(\mathbf r)$ be the component corresponding to $Z$ under Theorem~\ref{thm:tower}; thus $Z_1=Z$.

\subsection{Generic integral representations}
A representation over an algebraically closed extension $\Bbbk\supseteq\C$ is \emph{generic in $Z_k$} if its point $\op{Spec}\Bbbk\to Z_k$ has image the generic point of $Z_k$. Equivalently, evaluation of rational functions gives an embedding $\C(Z_k)\hookrightarrow\Bbbk$. We require this property at every central length, not merely membership in a chosen dense open subset of $Z_1$.

The dominant reduction maps $Z_{k+1}\dashrightarrow Z_k$ induce embeddings of function fields. Put
\begin{equation}\label{eq:field}
 \kappa_Z=\varinjlim_k\C(Z_k).
\end{equation}
We describe a generic lattice in coordinates. Let $n=N_d(\mathbf r)$ and choose the chart $\ms{X}_Z\subset\A^n_\OO$ of Proposition~\ref{prop:charts}, with coordinates $\xi_1,\ldots,\xi_n$. By \eqref{eq:jets}, this chart identifies $\kappa_Z$ with
\[ \C(\xi_{a,m}:1\le a\le n,\ m\ge0), \]
where the $\xi_{a,m}$ are algebraically independent. The assignments
\[ \xi_a\longmapsto\sum_{m\ge0}\xi_{a,m}t^m,\qquad 1\le a\le n, \]
define a $\kappa_Z[[t]]$-point of $\ms{X}_Z$: the constant term $f_Z(0,\xi_{1,0},\ldots,\xi_{n,0})$ is nonzero and hence invertible. Its reduction modulo $t^k$ is generic in $Z_k$ for every $k$.

Equivalently, in the fixed vertex power bases, the arrow matrices at the generic point of $Z_k$ have entries in $\C(Z_k)[t]/(t^k)$. After embedding these fields into $\kappa_Z$, the matrices are compatible under reduction. Their inverse limit defines a locally free $\PP_d$-module over $\kappa_Z[[t]]$. Its relations hold because they hold modulo every power of $t$. The same argument proves nilpotence: the ideal generated by the nonloop arrows has power $|\mathbf r|=\sum_i r_i$ acting as zero at every truncation, and therefore on the inverse limit.

Choose an algebraic closure $\Bbbk_Z$ of $\kappa_Z$. We write $\ms{M}_Z$ for the resulting lattice over $\Bbbk_Z[[t]]$ and
\[ K_Z=\Bbbk_Z((t)),\qquad M_Z^{\mr{gen}}=K_Z\otimes_{\Bbbk_Z[[t]]}\ms{M}_Z. \]
Thus $M_Z^{\mr{gen}}$ is a nilpotent module for the scalar extension of $\Gamma_d$ to $K_Z$. The coordinate construction above gives this representation after a coefficient-field embedding. Monodromy traces over $K_Z$ are taken with the compatible roots of $t$ fixed, as in Section~\ref{sec:induced}.

\begin{lemma}\label{lem:indep}
For a word $w$, define
\[ \lambda_Z^\mon(w)=F_w(M_Z^{\mr{gen}}). \]
This value is unchanged if $\ms{M}_Z$ is replaced by a locally free lattice over $\Bbbk[[t]]$, with $\Bbbk\supseteq\C$ algebraically closed, whose reduction modulo $t^k$ is generic in $Z_k$ for every $k$. It is independent of the chart, the algebraic closure, and further algebraically closed coefficient-field extensions.
\end{lemma}
\begin{proof}
Such a lattice gives compatible embeddings $\C(Z_k)\to\Bbbk$, hence an embedding $\kappa_Z\to\Bbbk$. Under this embedding, the universal arrow matrices agree with those of the chosen lattice. Extend it from $\Bbbk_Z$ to an algebraically closed coefficient field containing the two choices. The generic flag schemes are then identified by coefficient extension with $t$ and its compatible roots fixed. Their traces agree by \cite[Proposition~7.2]{FeiQP}. The definition of $\kappa_Z$ uses only the component reductions, not a chart.
\end{proof}

\subsection{The integral basis and faithfulness}
Corollary~\ref{thm:Serre} shows that $\lambda_Z^\mon$ annihilates the Serre ideal, also over the enlarged coefficient field $\Bbbk_Z$. We henceforth use the same notation for the resulting linear functional on $U^+_\Q$, extended by zero outside degree $\mathbf r$.

To explain its integral values, let $\mathbf i=(i_1,\ldots,i_v)$ and $\mathbf p=(p_1,\ldots,p_v)$, where each $p_a$ is a nonnegative integer and $\sum_a p_a\alpha_{i_a}=\mathbf r$. Denote by $\Fl^R_{\mathbf i,\mathbf p}(M)$ the \emph{balanced partial-flag scheme} whose successive quotients are $S_{i_1}^{\oplus p_1},\ldots,S_{i_v}^{\oplus p_v}$; quotients with $p_a=0$ are omitted. Here the simples and the flag scheme are understood over the coefficient field of $M$. The complete-flag refinement has trace $p!$, as established in Section~\ref{sec:hopf}; thus
\begin{equation}\label{eq:integral}
 \lambda_Z^\mon(e_{i_1}^{(p_1)}\cdots e_{i_v}^{(p_v)})
 =\chi_\sigma\!\left(\Fl^R_{\mathbf i,\mathbf p}(M_Z^{\mr{gen}})\right)\in\Z.
\end{equation}
The integrality is \cite[Corollary~6.5]{Nicaise}. Since these divided-power monomials generate $U^+_\Z$ additively, every element of $U^+_\Z$ has integral evaluation.

We use the triangular argument of Lusztig \cite{Lusztig}, in the form of \cite[Lemma~7.1]{GLSIV}. Choose an ordering of $I$, repeat it periodically, and take the first $s=|I|\,|\mathbf r|$ terms $i_1,\ldots,i_s$. Starting from a generic module in $Z$, remove the entire $i_a$-top at stage $a$, and let $p_a(Z)$ be its rank. Every canonical kernel is again crystal. A nonzero $E$-filtered module has a nonzero supported quotient, so each full cycle contains a positive removal until the module becomes zero. This occurs within $s$ terms. We put all subsequent $p_a(Z)$ equal to zero; for $\mathbf r=0$ we use the empty sequence and the monomial $1$.

The maximal-top correspondences recover $Z$ from $\mathbf p(Z)=(p_1(Z),\ldots,p_s(Z))$.
We call
\[ m_Z=e_{i_s}^{(p_s(Z))}\cdots e_{i_1}^{(p_1(Z))} \]
its \emph{adapted divided-power monomial}. The order is reversed because the first removed top is the last quotient of an ascending flag. We compare the sequences lexicographically: the first unequal entry determines their order. Use a chart constructed along this removal sequence. For any of its integral points, the incoming map at each removal stage is surjective onto the vertex space of the canonical kernel. Hence the generic fibre has the prescribed successive full-top ranks.

\begin{lemma}\label{lem:triangular}
For components $Z,Z'\in\Bcal_1(\mathbf r)$, we have
\[ \lambda_Z^\mon(m_Z)=1,\qquad \lambda_{Z'}^\mon(m_Z)=0\quad\text{if }\mathbf p(Z)>_{\mr{lex}}\mathbf p(Z'). \]
The same diagonal and vanishing identities hold if, for each component, we use any integral point of its adapted chart, including points over $\C[[t]]$.
\end{lemma}
\begin{proof}
A quotient supported at $i$ cannot have rank greater than the full $i$-top. A quotient with rank equal to that top has the canonical kernel. For unequal sequences, the preceding equal ranks force the same canonical removals, and the first larger requested rank makes the flag scheme empty. For equal sequences the partial flag is unique, and its divided-power evaluation equals one by \eqref{eq:integral}.
\end{proof}

Thus, if rows and columns are ordered by increasing $\mathbf p(Z)$, the matrix $(\lambda_Y^\mon(m_Z))_{Y,Z}$ is lower unitriangular: it has diagonal entries one and zeros above the diagonal.

\begin{theorem}[Integral monodromy basis]\label{thm:integralbasis}
The evaluation map
\begin{equation}\label{eq:integraliso}
 U^+_{\Z,\mathbf r}\longrightarrow\Z^{\Bcal_1(\mathbf r)},\qquad
 u\longmapsto(\lambda_Z^\mon(u))_Z
\end{equation}
is an isomorphism. For each $Z$, let $b_Z^\mon$ be the inverse image of the coordinate vector that is one at $Z$ and zero at the other components. The elements $b_Z^\mon$, in all degrees, form the \emph{monodromy basis} of $U^+_\Z$. They satisfy
\[ \lambda_Y^\mon(b_Z^\mon)=\delta_{Y,Z} \]
and do not depend on the periodic vertex order used to construct the adapted monomials.
\end{theorem}
\begin{proof}
By Lemma~\ref{lem:triangular} and \eqref{eq:integral}, the evaluation matrix on the $m_Z$ is integral unitriangular. Its size is $\dim_\Q U^+_{\Q,\mathbf r}$ by \eqref{eq:published}. Hence the $m_Z$ form a rational basis and the rational evaluation map is an isomorphism.

The inverse matrix is integral. Given $u\in U^+_{\Z,\mathbf r}$, its evaluations are integers by \eqref{eq:integral}. Applying that inverse matrix gives an integral combination of the $m_Z$ with the same evaluations as $u$. Rational injectivity makes this combination equal to $u$. The $m_Z$ are therefore an integral basis, and \eqref{eq:integraliso} is an integral isomorphism. The functionals were defined from the component reductions before choosing a periodic vertex order, so their dual basis has the same independence.
\end{proof}

Together with Theorem~\ref{thm:tower} and Lemma~\ref{lem:indep}, this proves Theorem~\ref{intro:basis}.

Let $\omega$ be the involutive anti-automorphism of $U^+_\Q$ fixing each $e_i$; it preserves $U^+_\Z$. By Section~\ref{sec:duality}, duality carries compatible generic lattices for $Z$ to those for $Z^*$. The ground-field flag isomorphism and Lemma~\ref{lem:indep} give
\begin{equation}\label{eq:duality-evaluations}
 \lambda_{Z^*}^{\mon}(u)=\lambda_Z^{\mon}(\omega(u)),\qquad
 \omega(b_Z^{\mon})=b_{Z^*}^{\mon}.
\end{equation}
The second identity follows from the defining duality of the basis. Independence of $\ell$ and of the tame generator is \cite[Corollary~6.5]{Nicaise}.

\begin{corollary}\label{thm:ramified-isomorphism}
The geometric convolution and evaluation maps are isomorphisms
\[ U^+_\Q\xrightarrow{\sim}\ms{C}_d\xrightarrow{\sim}\mc{M}_\sigma(\Gamma_d). \]
Together with Proposition~\ref{prop:coproduct}, this proves Theorem~\ref{intro:algebra}.
\end{corollary}
\begin{proof}
Both maps are surjective onto their stated targets by construction. To prove injectivity using modules over the original field $K=\C((t))$, fix a degree. In the special fibre of each adapted chart, choose a complex point and lift its coordinates to $\C[[t]]$. The defining polynomial remains invertible. Lemma~\ref{lem:triangular} gives a unitriangular evaluation matrix on the $m_Z$ for these actual $K$-modules. Since the $m_Z$ are a basis, these evaluations separate $U^+_{\Q,\mathbf r}$. This proves injectivity in every degree. The coproduct formula then identifies the geometric coproduct with the one for which every $e_i$ is primitive.
\end{proof}

\subsection{A calculation in type \texorpdfstring{$B_2$}{B2}}\label{sec:b2-example}
In this subsection, take
\[ C=\begin{pmatrix}2&-1\\-2&2\end{pmatrix},\qquad d=(2,1),\qquad R=\OO[x]/(x^2-t). \]
Thus the two vertex rings are $\OO_1=R$ and $\OO_2=\OO$. We compute the four basis elements in degree $(2,2)$. The relevant special-fibre modules are discussed in \cite[Sections~5.5, 8.2.2 and 8.2.5]{GLSIV}; the integral formulas below calculate their monodromy evaluations. In this degree the monodromy and special-fibre evaluations coincide on all convolution words. Write $a:1\to2$ and $b:2\to1$. The relations are $ba=0$ and $x(ab)+(ab)x=0$, together with $x^2=t$ at vertex $1$.

Let $\ms{E}$ have vertex modules $(R,0)$. The table defines four further lattices. Every arrow value not specified is zero. The maps are $\OO$-linear; a map from $uR$ is specified on the $\OO$-basis $u,ux$.
\begin{center}
\begin{tabular}{c|c|c|l}
\toprule
Lattice & Vertex $1$ & Vertex $2$ & Nonzero arrow values\\
\midrule
$\ms{T}_+$ & $uR$ & $\OO z$ & $zb=u$\\
$\ms{T}_-$ & $uR$ & $\OO z$ & $(ux)a=z$\\
$\ms{Q}_2$ & $uR$ & $\OO z_0\oplus\OO z_1$ & $z_0b=u,\quad(ux)a=z_1$\\
$\ms{Q}_1$ & $uR\oplus vR$ & $\OO z_0\oplus\OO z_1$ & $ua=z_0,\quad(ux)a=z_1,$\\
&&& $z_0b=v,\quad z_1b=-vx$\\
\bottomrule
\end{tabular}
\end{center}
These maps satisfy the relations. For example, on $u\in\ms{Q}_1$, the two terms of $x(ab)+(ab)x$ give $-vx$ and $vx$. Set
\[ \ms{M}_1=\ms{T}_-^{\oplus2},\qquad \ms{M}_2=\ms{Q}_1,\qquad \ms{M}_3=\ms{E}\oplus\ms{Q}_2,\qquad \ms{M}_4=\ms{T}_+^{\oplus2}. \]
All four lattices have rank $(2,2)$.

Their endomorphism modules are free of rank four over $\OO$, and the same equations hold after every central reduction. For $\ms{T}_\pm^{\oplus2}$ the endomorphism algebra is $\op{Mat}_2(\OO)$. For $\ms{E}\oplus\ms{Q}_2$, both cross Hom spaces vanish. The endomorphisms of $\ms{E}$ form $R$, while an endomorphism of $\ms{Q}_2$ has the form
\[ u\mapsto\alpha u,\qquad z_0\mapsto\alpha z_0+\beta z_1,\qquad z_1\mapsto\alpha z_1,\qquad \alpha,\beta\in\OO. \]
An endomorphism of $\ms{Q}_1$ is determined by
\[ u\longmapsto uf+vg,\qquad v\longmapsto vf(-x),\qquad f,g\in R, \]
where $f(-x)$ is the image of $f(x)\in R$ under the involution $x\mapsto-x$; the arrow equations determine its values on $z_0,z_1$.

At length $k$, the change-of-basis group is
$\op{GL}_2(R/t^kR)\times\op{GL}_2(\OO/t^k)$,
of dimension $12k$. The stabilizer is open in the $4k$-dimensional endomorphism space. Each orbit therefore has dimension $8k=kN_d(2,2)$ and closes to a maximal component. The four closures are distinct, since successive full-top removals in the order $1,2,1,2$ give
$(2,2,0,0),\ (1,2,1,0),\ (1,1,1,1),\ (0,2,2,0)$ respectively. All later entries are zero. The positive roots are $\alpha_1,\alpha_2,\alpha_1+\alpha_2,\alpha_1+2\alpha_2$. There are four ways to express $(2,2)$ as a nonnegative integer sum of these roots. The Poincar\'e--Birkhoff--Witt theorem and \eqref{eq:published} therefore show that these are all the maximal components. Denote their compatible labels by $Z_1,\ldots,Z_4$. A change of vertex bases with generic power-series coefficients gives a representation generic in each orbit at every length, without changing its flag schemes. Hence the traces of $\ms{M}_j$ compute $\lambda_{Z_j}^{\mon}$.

The only nonempty full flag schemes of the generic fibres of $\ms{E},\ms{T}_+,\ms{T}_-,\ms{Q}_2$ have words $1,12,21,212$, respectively, and each is a point. For $\ms{Q}_1$, injectivity of $b$ excludes a flag beginning with $2$, and surjectivity of $a$ excludes one ending with $2$. The only remaining word is $1221$. Its first submodule is $vK_1$, and the intermediate line at vertex $2$ is arbitrary. This flag scheme is $\mb{P}^1_K$, whose monodromy trace is two. Proposition~\ref{prop:coproduct} determines the direct-sum evaluations.

The adapted divided-power monomials are
\[ m_1=e_2^{(2)}e_1^{(2)},\qquad m_2=e_1e_2^{(2)}e_1,\qquad m_3=e_2e_1e_2e_1,\qquad m_4=e_1^{(2)}e_2^{(2)}. \]
With rows indexed by $Z_1,\ldots,Z_4$ and columns by $m_1,\ldots,m_4$, the evaluation matrix is
\begin{equation}\label{eq:b2-matrix}
 \bigl(\lambda_{Z_j}^{\mon}(m_l)\bigr)_{j,l=1}^{4}
 =\begin{pmatrix}1&0&2&0\\0&1&0&0\\0&0&1&0\\0&0&0&1\end{pmatrix}.
\end{equation}
For instance, there are two ways to interleave two copies of $21$, preserving the order in each, to obtain $2121$. Thus $\lambda_{Z_1}^{\mon}(m_3)=2$. Inverting \eqref{eq:b2-matrix} gives
\[ b_{Z_1}^{\mon}=m_1,\qquad b_{Z_2}^{\mon}=m_2,\qquad b_{Z_3}^{\mon}=m_3-2m_1=e_2e_1^{(2)}e_2,\qquad b_{Z_4}^{\mon}=m_4. \]
The alternative expression for $b_{Z_3}^{\mon}$ follows by multiplying the Serre relation $e_1^2e_2-2e_1e_2e_1+e_2e_1^2=0$ on the left by $e_2$. The example thus determines the basis elements, including the integral normalization of the nontrivial correction.

\section{Crystal compatibility and integrable modules}\label{sec:perfect}
We return to arbitrary $(C,d)$. For $Z\in\Bcal_1(\mathbf r)$, let $t_i(Z)$ and $s_i(Z)$ be the ranks of $\fac_i(M)$ and $\op{sub}_i(M)$ for a generic module $M$ in $Z$. These are the full vertex-supported top and socle defined in Section~\ref{sec:setup}. Their ranks are constant on a dense open of $Z$.

Write $\mathsf K_iZ$ for the component obtained by removing the entire $i$-top, and $\mathsf Q_iZ$ for the component obtained by quotienting by the entire $i$-socle. The GLS correspondences give
\[ \mathsf K_iZ\in\Bcal_1(\mathbf r-t_i(Z)\alpha_i),\qquad \mathsf Q_iZ\in\Bcal_1(\mathbf r-s_i(Z)\alpha_i). \]
The first component has zero generic $i$-top, and the second has zero generic $i$-socle. If the relevant rank is zero, the corresponding removal leaves $Z$ unchanged. We use zero for a homogeneous component of $U^+_\Q$ whose degree has a negative coordinate.

Throughout this section we prove the top/right-multiplication statements; the socle/left-multiplication statements follow from \eqref{eq:duality-evaluations} and
$s_i(Z)=t_i(Z^*)$, $\mathsf Q_iZ=(\mathsf K_iZ^*)^*$, by Section~\ref{sec:duality}.

\subsection{Canonical removals and multiplication}
We explain why the canonical removals of $\ms{M}_Z$ compute the functionals of these smaller components. By Theorem~\ref{thm:tower}, all generic truncations have the same free top and socle ranks. Smith normal form over $\Bbbk_Z[[x_i]]$, with $t=x_i^{d_i}$, then gives only unit and zero invariant factors for the integral incoming and outgoing maps. Indeed, a nonzero factor $x_i^a$ with $a>0$ would produce a nonfree top or socle at every central length $k$ with $kd_i>a$.

The integral incoming and outgoing maps therefore split onto their images, and their kernels and cokernels commute with reduction. In the proof of Theorem~\ref{thm:tower}, the charts beginning with the prescribed removal have dominant projections to the smaller components. Hence the full-top kernel and the full-socle quotient of the generic lattice represent $\mathsf K_iZ$ and $\mathsf Q_iZ$, respectively, at every central length. Lemma~\ref{lem:indep} identifies their evaluations with the canonical generic evaluations.

The following is the monodromy version of the full-top restriction in \cite[Lemma~7.1]{GLSIV}.
\begin{lemma}\label{lem:removal}
Let $Z\in\Bcal_1(\mathbf r)$, $p=t_i(Z)$, and $q=s_i(Z)$. Then
\[ \lambda_Z^\mon(u e_i^{(p)})=\lambda_{\mathsf K_iZ}^\mon(u),\qquad u\in U^+_{\Q,\mathbf r-p\alpha_i}, \]
and
\[ \lambda_Z^\mon(e_i^{(q)}v)=\lambda_{\mathsf Q_iZ}^\mon(v),\qquad v\in U^+_{\Q,\mathbf r-q\alpha_i}. \]
If $a>p$, then $\lambda_Z^\mon(u e_i^{(a)})=0$ for $u\in U^+_{\Q,\mathbf r-a\alpha_i}$. If $a>q$, then $\lambda_Z^\mon(e_i^{(a)}u)=0$ in the same degree.
\end{lemma}
\begin{proof}
A supported quotient of rank $p$ has the canonical full-top kernel. Its divided-power normalization is one by \eqref{eq:integral}, while a quotient of greater rank is impossible. The genericity of this canonical removal identifies the remaining evaluations with those of $\mathsf K_iZ$.
\end{proof}

We denote the raising operators of the two component crystals by $\mathsf F_i^{\mr{top}}$ and $\mathsf F_i^{\mr{soc}}$. Concretely, $\mathsf F_i^{\mr{top}}Z$ is the unique component of rank $\mathbf r+\alpha_i$ with
\[ \mathsf K_i(\mathsf F_i^{\mr{top}}Z)=\mathsf K_iZ,\qquad t_i(\mathsf F_i^{\mr{top}}Z)=t_i(Z)+1. \]
Similarly, $\mathsf F_i^{\mr{soc}}Z$ has the same full-socle quotient component as $Z$ and has socle rank $s_i(Z)+1$. These are the component correspondences of \cite[Section~5.3]{GLSIV}. Their $a$th powers mean $a$ successive applications, with the zeroth power the identity. The spans in the next theorem range over components in all degrees.

\begin{theorem}[Two-sided crystal compatibility]\label{thm:perfect}
For every $i\in I$ and integer $a\ge0$,
\begin{equation}\label{eq:ideals}
\begin{aligned}
 U^+_\Q e_i^{(a)}&=\op{span}_\Q\{b_Z^\mon:t_i(Z)\ge a\},\\
 e_i^{(a)}U^+_\Q&=\op{span}_\Q\{b_Z^\mon:s_i(Z)\ge a\}.
\end{aligned}
\end{equation}
The first subspace is a left ideal and the second a right ideal. If $Z$ has rank $\mathbf r$ and $p=t_i(Z)$, then
\begin{equation}\label{eq:leading}
 b_Z^\mon e_i^{(a)}=
 \binom{p+a}{a}b_{(\mathsf F_i^{\mr{top}})^aZ}^\mon
 +\sum_{t_i(Y)>p+a}c_Yb_Y^\mon,\qquad c_Y\in\Z,
\end{equation}
where the sum is over $Y\in\Bcal_1(\mathbf r+a\alpha_i)$. The corresponding formula for left multiplication is
\[ e_i^{(a)}b_Z^\mon=\binom{s_i(Z)+a}{a}b_{(\mathsf F_i^{\mr{soc}})^aZ}^\mon+\sum_{s_i(Y)>s_i(Z)+a}c'_Yb_Y^\mon,\qquad c'_Y\in\Z, \]
with the same rank for $Y$.
\end{theorem}
\begin{proof}
The vanishing in Lemma~\ref{lem:removal} puts $U^+_\Q e_i^{(a)}$ in the span on the right of the first identity. The Poincar\'e--Birkhoff--Witt theorem gives an associated graded algebra equal to the symmetric algebra of $\mf{n}(C)$, which is a domain. Thus multiplication by $e_i^a$ is injective. In degree $\mathbf r$, its source has dimension $\dim U^+_{\Q,\mathbf r-a\alpha_i}$. The map $Z\mapsto(\mathsf F_i^{\mr{top}})^aZ$ is a bijection from components of that smaller rank to the components of rank $\mathbf r$ with $t_i\ge a$. Equation~\eqref{eq:published} therefore gives equality of dimensions and proves the first identity.

Put $Z_0=\mathsf K_iZ$, so $t_i(Z_0)=0$. Full-top recursion gives
\[ b_{Z_0}^\mon e_i^{(p)}\equiv b_Z^\mon\pmod{U^+_\Q e_i^{(p+1)}}. \]
Indeed, its evaluations on components with top rank less than $p$ vanish, and those with top rank $p$ are zero except at $Z$, where the value is one. Multiplying the congruence on the right by $e_i^{(a)}$ and using
$e_i^{(p)}e_i^{(a)}=\binom{p+a}{a}e_i^{(p+a)}$
gives the leading term in \eqref{eq:leading}. The error belongs to $U^+_\Q e_i^{(p+a+1)}$, whose basis vectors are identified by \eqref{eq:ideals}. All coefficients are integers because both the divided powers and the monodromy basis belong to $U^+_\Z$.
\end{proof}

\subsection{Integral strings and representation lattices}\label{sec:saturation}
For $a\ge0$, define additive subgroups of $U^+_\Z$ by
\[ \ms{F}_i^a=U^+_\Z\cap(U^+_\Q e_i^{(a)})=\bigoplus_{t_i(Z)\ge a}\Z b_Z^\mon,\qquad U^+_{i,0}=\bigoplus_{t_i(Z)=0}\Z b_Z^\mon. \]
The equality for $\ms{F}_i^a$ follows from \eqref{eq:ideals} and Theorem~\ref{thm:integralbasis}. These groups form a decreasing filtration. In each degree, $\ms{F}_i^a$ is \emph{saturated}: the quotient of the ambient free abelian group by it has no nonzero torsion.

For integrable representations we use left modules over the Kac--Moody algebra $\mf{g}(C)$. Let $\mf{h}$ be its Cartan subalgebra, $h_i\in\mf{h}$ the simple coroots, and $f_i$ the negative Chevalley generators. A weight $\lambda\in\mf{h}^*$ is \emph{dominant integral} if $\lambda_i:=\lambda(h_i)$ is a nonnegative integer for every $i$. Let $L(-\lambda)$ be the irreducible integrable lowest-weight module over $\Q$ with lowest weight $-\lambda$. Here integrability means that each $e_i$ and $f_i$ acts locally nilpotently. Choose a nonzero lowest-weight vector $v_{-\lambda}$; thus $f_iv_{-\lambda}=0$ and $hv_{-\lambda}=-\lambda(h)v_{-\lambda}$ for $h\in\mf{h}$. Its \emph{cyclic Kostant lattice} is
\[ L_\Z(-\lambda)=U^+_\Z v_{-\lambda}\subset L(-\lambda). \]
The map from $U^+_\Z$ to this lattice sends $u$ to $uv_{-\lambda}$.

\begin{proposition}\label{prop:integralstrings}
Multiplication gives an isomorphism of graded abelian groups
\[ \bigoplus_{p\ge0}U^+_{i,0}\otimes_\Z\Z e_i^{(p)}\xrightarrow{\sim}U^+_\Z,\qquad v\otimes e_i^{(p)}\longmapsto ve_i^{(p)}, \]
where $\Z e_i^{(p)}$ is the free rank-one subgroup generated by $e_i^{(p)}$. Moreover,
\begin{equation}\label{eq:saturatedideal}
 \ms{F}_i^a=\sum_{p\ge a}U^+_\Z e_i^{(p)}.
\end{equation}
For every dominant integral $\lambda$,
\begin{equation}\label{eq:allpowerskernel}
 \ker\bigl(U^+_\Z\longrightarrow L_\Z(-\lambda)\bigr)=\sum_i\sum_{p>\lambda_i}U^+_\Z e_i^{(p)}.
\end{equation}
The vectors $b_Z^\mon v_{-\lambda}$ with $t_i(Z)\le\lambda_i$ for every $i$ form a $\Z$-basis of $L_\Z(-\lambda)$. There are analogous decompositions with the divided powers on the left, using the socle filtration.
\end{proposition}
\begin{proof}
For $t_i(Z_0)=0$, \eqref{eq:leading} gives
\[ b_{Z_0}^\mon e_i^{(p)}=b_{(\mathsf F_i^{\mr{top}})^pZ_0}^\mon+\sum_{t_i(Y)>p}c_Yb_Y^\mon,\qquad c_Y\in\Z. \]
Every label is determined by its zero-top predecessor and its top rank. In each degree these products therefore give an integral unitriangular change of basis. This proves the direct-sum decomposition. Restricting to labels with top rank at least $a$ proves \eqref{eq:saturatedideal}.

The standard rational presentation is \cite[Chapter~10]{Kac}
\[ L(-\lambda)\cong U^+_\Q\Big/\sum_iU^+_\Q e_i^{(\lambda_i+1)}. \]
By \eqref{eq:ideals}, the denominator is spanned by exactly those basis elements for which $t_i(Z)>\lambda_i$ for at least one $i$. Its intersection with $U^+_\Z$ is the integral span of the same elements. Equation~\eqref{eq:saturatedideal} identifies this subgroup with the right side of \eqref{eq:allpowerskernel}. The other labels give the claimed basis of the quotient.
\end{proof}

The \emph{Chevalley involution} of $\mf{g}(C)$ sends $e_i$ to $f_i$, $f_i$ to $e_i$, and $h\in\mf{h}$ to $-h$. Applying it to the preceding construction gives the corresponding basis of an integrable highest-weight lattice.

All sums above are finite in each homogeneous degree. The sum over every $p>\lambda_i$ in \eqref{eq:allpowerskernel} is necessary: in rank one and degree $n\ge a$,
\[ (U^+_\Z e_i^{(a)})_n=\Z\binom na e_i^{(n)}, \]
whereas its saturation is $\Z e_i^{(n)}$. Finally, the quotient in \eqref{eq:allpowerskernel} is free. For any commutative coefficient ring $B$, extension of scalars therefore gives a $B$-basis of $B\otimes_\Z L_\Z(-\lambda)$ with the same labels and the induced divided-power action. In positive characteristic this module need not be irreducible.

\section{Changes of coordinates}\label{sec:invariance}
The following transport statement applies to the coordinate changes used in this section. Let $\varphi:\ms{P}\to\ms{P}'$ be an isomorphism between two copies of the standard integral algebra for $(C,d)$, possibly with different coordinates or orientation signs. If $M$ is a right $\ms{P}$-module, its transport to $\ms{P}'$ has action $m\cdot a'=m\varphi^{-1}(a')$; when $\varphi$ is semilinear, the scalar structure is transported as well. We write $\varphi_*Z$ for the induced component label.

\begin{proposition}\label{prop:presentation-transport}
Assume that $\varphi$ preserves the vertex idempotents and carries each prescribed vertex ring onto its counterpart. It may be \emph{semilinear} for a $\C$-algebra automorphism $\phi$ of $\OO$, meaning that $\varphi(ca)=\phi(c)\varphi(a)$ for $c\in\OO$. Then it identifies the component systems and their flag schemes after the corresponding change of scalars. For $u\in U^+_\Q$,
\[ \lambda_{\varphi_*Z}^{\mon}(u)=\lambda_Z^{\mon}(u),\qquad b_{\varphi_*Z}^{\mon}=b_Z^{\mon}. \]
The left sides are computed for $\ms{P}'$ and the right sides for $\ms{P}$. The basis equality uses the identifications with $U^+_\Q$ that send the vertex-simple functions to the same generators $e_i$.
\end{proposition}
\begin{proof}
Twisting by $\varphi^{-1}$ preserves precisely the stable vertex submodules and the ranks of their successive quotients. This is functorial over parameter rings, with scalar structure changed by $\phi$ when necessary. It identifies representation schemes, flag schemes, and central reductions. Hence it transports maximal components, their dominant reduction maps, and their compatible generic points.

Every automorphism of $\OO=\C[[t]]$ sends $t$ to $t$ times a unit. Extending it to the chosen algebraic closures identifies the absolute Galois groups and carries a tame generator to a tame generator. The generic flag cohomology is identified under this transport, and its trace is unchanged. This proves the functional equality; the defining duality of the bases proves the other one.
\end{proof}

\subsection{Common scaling of the symmetrizer}
Fix an integer $m\ge1$. We write $\Gamma_{md}(t)$ to indicate that the tuple is $md=(md_i)_{i\in I}$ and the base parameter is $t$. In this algebra, put
\[ u=\sum_i x_i^{d_i},\qquad K_t=\C((t)). \]
The edge commutation relation raised to the $e_{ij}$th power gives $x_i^{d_i}a_{ij}=a_{ij}x_j^{d_j}$. Thus $u$ is central and satisfies $u^m=t$. The polynomial $X^m-t$ is irreducible over $K_t$ by Eisenstein's criterion, so $u$ defines a central subfield $K_u=\C((u))$. As a $K_u$-algebra, $\Gamma_{md}(t)$ is $\Gamma_d(u)$: on the $i$th vertex, $u$ acts as $x_i^{d_i}$.

Since $t=u^m$, reduction modulo $t^k$ for the tuple $md$ agrees with reduction modulo $u^{mk}$ for the tuple $d$. Thus, if $(Z_k)_{k\ge1}$ is a component system for $d$, the corresponding system for $md$ is $(Z_{mk})_{k\ge1}$. Passing to these indices leaves the direct-limit field \eqref{eq:field} unchanged. The compatible generic matrices describe the same representation, now considered over $K_t\subset K_u$. Its balanced flag scheme over $K_t$ is the restriction of scalars of its balanced flag scheme over $K_u$.

Choose compatible tame generators with $\sigma_u=\sigma_t^m$. Applying \eqref{eq:weiltrace}, including after the coefficient extension used for a generic representation, gives
\begin{equation}\label{eq:scale}
 \lambda_{Z_m,md}^\mon=\lambda_{Z,d}^\mon,\qquad
 b_{Z_m,md}^\mon=b_{Z,d}^\mon.
\end{equation}
The tuple following the comma indicates which ramification data define the evaluation. For connected $C$, every integral symmetrizer is a positive integer multiple of the primitive one, whose entries have greatest common divisor one. Hence the basis is independent of that multiple. For disconnected $C$, apply the argument separately on each connected component; the flag schemes decompose accordingly and their traces satisfy the K\"unneth formula.

\subsection{Orientation signs and arrow bases}
Changing an orientation sign is implemented by negating one arrow in the corresponding doubled pair. This is an integral isomorphism preserving the vertex rings, so Proposition~\ref{prop:presentation-transport} applies. Negating the other arrow differs by a simultaneous rescaling of the pair. These rescalings form a connected multiplicative group and therefore fix every irreducible component. The two choices give the same correspondence of labels.

Fix an edge $\{i,j\}$ and put $g=g_{ij}$. List its forward arrows, in the chosen orientation, as a column $a=(a_1,\ldots,a_g)^{\mathsf T}$, and the paired reverse arrows as $b=(b_1,\ldots,b_g)^{\mathsf T}$. For $G\in\op{GL}_g(\OO_{ij})$, replace them by
$a\mapsto Ga$ and $b\mapsto(G^{-1})^{\mathsf T}b$.
Here $\mathsf T$ denotes matrix transpose, and the common edge ring $\OO_{ij}$ acts through its prescribed embeddings in the two vertex rings. Its coefficients commute through the arrows by \eqref{eq:comm}. The change preserves $\sum_ca_cb_c$ and the corresponding sum at the other endpoint, so every moment relation is preserved. A vertex flag is stable under the new arrows if and only if it is stable under the old ones, since $G$ is invertible over the same coefficient ring.

At length $k$, these changes form $\op{GL}_g(\C[z]/(z^{ke_{ij}}))$, where $z$ is the parameter of $\OO_{ij}$ and $t=z^{e_{ij}}$. Reduction modulo $z$ has connected affine fibres over $\op{GL}_g(\C)$, so this group is connected. A connected algebraic group fixes every irreducible component. Thus these changes leave each component label unchanged, and Proposition~\ref{prop:presentation-transport} gives equality of the labelled bases.

\subsection{Formal changes of uniformizer}
Let $\phi$ be the automorphism of $\OO$ defined by
\[ \phi(t)=t\eta(t),\qquad \eta(t)\in\C[[t]]^\times. \]
Put $L=\lcm_i d_i$, $w_i=L/d_i$, and $h_{ij}=L/e_{ij}$. Hensel's lemma gives a unit $v(t)$ with $v(t)^L=\eta(t)$, since the constant coefficient has an $L$th root in $\C$ and $L$ is invertible. Use the forward arrow of each doubled pair fixed in the preceding subsection. Define
\begin{equation}\label{eq:coordinate}
 t\mapsto\phi(t),\qquad x_i\mapsto v(t)^{w_i}x_i,\qquad
 a_{\mr{forward}}\mapsto a_{\mr{forward}},\qquad
 a_{\mr{reverse}}\mapsto v(t)^{-h_{ij}}a_{\mr{reverse}},
\end{equation}
where the exponent $h_{ij}$ belongs to the edge containing the arrow. These assignments define a $\phi$-semilinear isomorphism. Indeed, $d_iw_i=L$ verifies the power relations, and $f_{ji}w_i=f_{ij}w_j=h_{ij}$ verifies the commutation relations. Each summand in the $i$-moment relation is multiplied by $v^{-w_i}$, since $(f_{ji}-1)w_i-h_{ij}=-w_i$. The inverse uses $\phi^{-1}$ and the unit $(v\circ\phi^{-1})^{-1}$.

The isomorphism preserves every ideal $(t^k)$. On the special fibre, replacing $v(0)$ in \eqref{eq:coordinate} by a variable $z\ne0$ defines an action of the multiplicative group $\Gm=\op{Spec}\C[z,z^{-1}]$. Changing the vertex power bases identifies the rescaled loop matrices with the prescribed ones. This gives a connected-group action on the fixed-vertex representation variety and fixes each component. Proposition~\ref{prop:presentation-transport} therefore preserves the labelled basis. Any other choice of $v$ differs by a constant $L$th root of unity in the same action, so the result is independent of this choice.

\section{The symmetric case and direct sums}\label{sec:applications}
\subsection{The symmetric case}
For an ordinary preprojective module $M$, the Euler characteristics of its composition-series varieties define a linear form on $U^+_\Q$. These are the linear forms encoded by the GLS functions $\varphi_M$ \cite[Section~1.2]{GLSJAMS}. Their descending series $M=X_0\supseteq\cdots\supseteq X_h=0$ becomes our ascending series with successive quotients $X_{h-1}/X_h,\ldots,X_0/X_1$. We use this correspondence of conventions throughout the comparison.

\begin{proposition}\label{prop:symmetric}
If $C$ is symmetric, the monodromy basis agrees with Lusztig's semicanonical basis in $U^+_\Z$. More precisely, if $f_Z$ denotes Lusztig's semicanonical basis element labelled by the component $Z$ in his convolution convention, then $f_Z=b_{Z^*}^\mon$.
\end{proposition}
\begin{proof}
First take $d_i=1$ for every $i$, so $D$ is the identity matrix. The relations $x_i=t1_i$ identify $\PP_d$ with the scalar extension to $\OO$ of the ordinary preprojective algebra. Let $Z\in\Bcal_1(\mathbf r)$. Since the remaining relations have coefficients in $\C$, the chart of Proposition~\ref{prop:charts} can be chosen as the scalar extension of a complex chart of $Z$.

Apply the generic formal-arc construction of Section~\ref{sec:basis} to this chart over an algebraically closed coefficient field $\Bbbk\supseteq\C$. It gives a morphism
\[ \gamma:\op{Spec}\Bbbk[[t]]\longrightarrow Z \]
whose special point maps to the generic point of $Z$. Its reductions are generic in the corresponding components, so Lemma~\ref{lem:indep} allows us to use this lattice to compute $\lambda_Z^\mon$.

Fix a word $w$ of rank $\mathbf r$, and let $p_w:\mc{F}_w\to Z$ be the projective family of ordinary representation flags of type $w$. There is a dense open $Z^\circ\subseteq Z$ on which all the sheaves $R^q(p_w)_*\Q_\ell$ are \emph{lisse}, that is, locally constant constructible sheaves in the \'etale topology. The morphism $\gamma$ factors through $Z^\circ$, since its special point does. Their pullbacks to $\op{Spec}\Bbbk[[t]]$ are constant: this complete local ring has algebraically closed residue field and is strictly henselian. Thus inertia acts trivially on the generic flag cohomology. If $\overline\eta$ is a geometric generic point of $Z$, proper base change gives
\[ \lambda_Z^\mon(w)=\chi_c((\mc{F}_w)_{\overline\eta}). \]
 Lusztig's convolution is quotient-first \cite[Section~5.2]{GLS2005}, so its generic functional at $Z$ is $\lambda_Z^\mon\circ\omega$. Equation~\eqref{eq:duality-evaluations} therefore identifies its dual basis element with $\omega(b_Z^\mon)=b_{Z^*}^\mon$. For connected symmetric $C$, every symmetrizer is a common multiple of the identity, so \eqref{eq:scale} proves the result for all $d$. For disconnected $C$, apply the same argument to each connected component.
\end{proof}

\subsection{Multiplication under direct sum}
We work in the \emph{graded dual}
\[ (U^+_\Q)^\circ=\bigoplus_{\mathbf r\in\mb{N}^I}\Hom_\Q(U^+_{\Q,\mathbf r},\Q). \]
Its multiplication is dual to the coproduct: $(\lambda\mu)(u)=(\lambda\otimes\mu)(\Delta u)$. We extend a homogeneous functional by zero in the other degrees. The monodromy functionals $\lambda_Z^\mon$ belong to this dual.

Let $Z\in\Bcal_1(\mathbf r)$ and $W\in\Bcal_1(\mathbf s)$. We call them \emph{generically Ext-orthogonal} if $\Ext^1_{\Pi_1}(M,N)=\Ext^1_{\Pi_1}(N,M)=0$ on a dense open subset of $Z\times W$. For $k\ge1$ and a rank vector $\mathbf a$, write
\[ G_k(\mathbf a)=\prod_i\op{GL}_{a_i}(H_{i,k}). \]
Using fixed decompositions at the vertices, let $Z\oplus W$ denote the closure of the union of the $G_1(\mathbf r+\mathbf s)$-orbits of $M\oplus N$ for $(M,N)\in Z\times W$. This is an irreducible subvariety. The next proposition extends the classical multiplicative property of the dual semicanonical basis \cite[Theorem~1.1]{GLS2005}.

\begin{proposition}\label{prop:Ext-product}
If $Z$ and $W$ are generically Ext-orthogonal, then $Z\oplus W$ is a maximal component and
$\lambda_Z^\mon\lambda_W^\mon=\lambda_{Z\oplus W}^\mon$ in $(U^+_\Q)^\circ$.
\end{proposition}
\begin{proof}
Choose locally free lattices $\ms{M},\ms{N}$ over a common algebraically closed coefficient field $\Bbbk\supseteq\C$, with independent generic coordinates for the charts of $Z$ and $W$. Put $M_k=\ms{M}/t^k\ms{M}$ and $N_k=\ms{N}/t^k\ms{N}$. All Hom and Ext dimensions in the proof are over $\Bbbk$, for the corresponding scalar extension of $\Pi_k$. Write $\OO_{i,\Bbbk}=\Bbbk[[t]][x_i]/(x_i^{d_i}-t)$, and similarly define $\OO_{ij,\Bbbk}$.

A module homomorphism is a tuple $f_i\in\Hom_{\OO_{i,\Bbbk}}(\ms{M}_i,\ms{N}_i)$ satisfying
$f_j\circ M_a=N_a\circ f_i$ for $a:i\to j$,
where $M_a,N_a$ are the arrow maps and $a$ ranges over all nonloop arrows, with their multiplicities. The differences on the two sides lie in $\Hom_{\OO_{ij,\Bbbk}}(\ms{M}_i,\ms{N}_j)$. Thus these equations define a linear map between finite free $\Bbbk[[t]]$-modules, whose reduced kernel is $\Hom(M_k,N_k)$ at length $k$. A nonzero Smith factor $t^v$ contributes $\min(k,v)$ to this kernel dimension, while a zero factor contributes $k$. Since $\min(k,v)\le k\min(1,v)$, we obtain $\dim\Hom(M_k,N_k)\le k\dim\Hom(M_1,N_1)$,
and the analogous inequality with $M,N$ interchanged.

Define $(\mathbf r,\mathbf s)_d=\mathbf r^{\mathsf T}DC\mathbf s$. The GLS Hom--Ext formula \cite[Lemma~2.1]{GLSIV} gives
\[ \dim\Ext^1(M_k,N_k)=\dim\Hom(M_k,N_k)+\dim\Hom(N_k,M_k)-k(\mathbf r,\mathbf s)_d. \]
At $k=1$ the Ext term is zero. The two Hom bounds therefore make the right side nonpositive for every $k$. Nonnegativity forces equality in both bounds and $\Ext^1(M_k,N_k)=0$; the same holds in the opposite direction.

Let $Z_k\oplus W_k$ be the direct-sum closure defined using $G_k$ in place of $G_1$. The direct-sum dimension formula of \cite[Section~1]{CBS}, in these fixed vertex coordinates, is
\[ \dim(Z_k\oplus W_k)=\dim Z_k+\dim W_k+2k\sum_i d_ir_is_i-\dim\Hom(M_k,N_k)-\dim\Hom(N_k,M_k). \]
The group contribution is $\dim G_k(\mathbf r+\mathbf s)-\dim G_k(\mathbf r)-\dim G_k(\mathbf s)=2k\sum_i d_ir_is_i$. Indeed, $G_k(\mathbf a)$ is an open subset of $\prod_i\End_{H_{i,k}}(H_{i,k}^{a_i})$, of dimension $k\sum_i d_i a_i^2$. Substitution of the Hom--Ext formula and \eqref{eq:published} gives
\[ \dim(Z_k\oplus W_k)=kN_d(\mathbf r+\mathbf s)-\dim\Ext^1(M_k,N_k)=kN_d(\mathbf r+\mathbf s). \]
Thus each direct-sum closure is maximal. Reduction is compatible with direct sum, so Theorem~\ref{thm:tower} identifies these as the component system of $Z\oplus W$.

Choose independent generic power-series coordinates for a vertex change of basis. At each length, the direct-sum map from the product of the two component charts and the vertex change-of-basis group is dominant onto the direct-sum closure. These choices therefore give a lattice generic in $Z_k\oplus W_k$ for every $k$. Its flag evaluations are unchanged by the vertex change of basis. Proposition~\ref{prop:coproduct} now gives the required identity of functionals. If $Z=W$, we use two independent generic members of the component, not two copies of one fixed module.
\end{proof}

\section{The Serre quotient and generic evaluation}\label{sec:comparison}
We first identify the ordinary Serre quotient, then use one affine module to compare special-fibre generic evaluation with monodromy.

\subsection{The canonical Serre quotient}\label{sec:canonical-quotient}
For arbitrary $(C,d)$, let $\mc{A}=\widetilde{\mc{M}}_\Q(\Pi_1)$ be the ordinary Euler-characteristic convolution algebra generated by $T_i=1_{E_i}$. This is the rational form of GLS's $\widetilde{\mc{M}}(\Pi)$; their notation $\mc{M}(\Pi)$ denotes its quotient by the Serre ideal. With ascending flags,
\[ (T_{i_1}\cdots T_{i_h})(M)=\chi_c\Fl^{\lf}_{i_1\cdots i_h}(M). \]
Let $\ev_\Pi:T_\Q\twoheadrightarrow\mc{A}$ send $z_i$ to $T_i$, and put $\mc{I}=\ev_\Pi(J_C)$. The direct-sum coproduct satisfies
\[ (\Delta g)(M,N)=g(M\oplus N),\qquad \Delta(T_i)=T_i\otimes1+1\otimes T_i; \]
see \cite[Section~4.3 and Propositions~4.5 and~4.7]{GLSIII} for the direct-sum formula for $H$, and \cite[Section~2.5]{GLSIV} for the preprojective coproduct. Here $1=1_0$. Since commutators of primitive elements are primitive, $\mc{I}$ is a Hopf ideal. The Serre relations therefore give a canonical surjection $\eta_\Pi:U^+_\Q\twoheadrightarrow\mc{A}/\mc{I}$.

\begin{theorem}[The canonical Serre quotient]\label{thm:canonical-quotient}
For every $(C,d)$, the map
\[ \eta_\Pi:U^+_\Q\xrightarrow{\sim}\mc{A}/\mc{I},\qquad e_i\longmapsto T_i+\mc{I}, \]
is an isomorphism of graded Hopf algebras. For any acyclic orientation $\Omega$, restriction to modules over $H=H(C,D,\Omega)^{\mr{op}}$ gives a surjective Hopf homomorphism
\[ \op{res}_\Omega:\mc{A}\longrightarrow\mc{M}_\Q(H),\qquad \ker\op{res}_\Omega=\mc{I}. \]
\end{theorem}
\begin{proof}
The quotient $\Pi_1\twoheadrightarrow H$ sets the opposite arrows to zero. Every locally free $H$-module is $E$-filtered: order vertices with arrows pointing to earlier vertices, then filter by vertex sums and refine the free factors. Inflation preserves submodules and flags. Hence restriction preserves convolution and coproduct, sends $T_i$ to $\theta_i$, and is surjective.

The acyclic Serre relations \cite[Proposition~4.9]{GLSIII} give $\mc{I}\subseteq\ker\op{res}_\Omega$. By Theorem~\ref{thm:acyclic}, the composite
\[ U^+_\Q\xrightarrow{\eta_\Pi}\mc{A}/\mc{I}\xrightarrow{\overline{\op{res}}_\Omega}\mc{M}_\Q(H) \]
is an isomorphism. Thus both arrows are isomorphisms, proving the kernel equality.
\end{proof}
After scalar extension, this proves \cite[Theorem~7.4(i)]{GLSIV} without the Support Conjecture. The proof uses only the acyclic realization in Appendix~\ref{sec:acyclic}, not generic component evaluations. Theorem~\ref{thm:affine-support} shows that the equivalence in \cite[Theorem~7.4(iii)]{GLSIV} fails without the Support Conjecture; it does not contradict that conditional theorem. The basis assertion in part~(ii) is a separate question. Corollary~\ref{cor:generic-kernel-not-ideal} shows that generic values alone do not determine quotient classes.

\subsection{A counterexample to the Support Conjecture}\label{sec:affine-support}
For the rest of this section, use $(C_*,d_*)$ from Theorem~\ref{intro:counterexample}. Its quadratic form is
$q_{d_*}(r_1,r_2,r_3)=(r_1-r_2)^2+(r_3-r_2)^2$,
so $C_*$ is affine with null root $(1,1,1)$. Put $S_{21}=(\op{ad}T_2)^2(T_1)$ and
\[ f=(\op{ad}z_2)^2(z_1)\,z_1z_3z_2z_3\in J_{C_*},\qquad F=\ev_\Pi(f)=S_{21}T_1T_3T_2T_3. \]
Its rank degree is $\mathbf r=(2,3,2)$, with ordinary dimension vector $\mathbf v=(2,6,2)$. To compare with \cite[Conjecture~1.4]{GLSIV}, we allow the loops to vary and use $G(\mathbf v)=\op{GL}_2(\C)\times\op{GL}_6(\C)\times\op{GL}_2(\C)$. Support means the constructible nonzero locus, with the dimension of its closure. In type $B_2$, \cite[Section~8.2.4]{GLSIV} already explains the role of Euler-characteristic cancellation and the failure of non-maximal support to define an ideal in the larger function algebra.

\begin{theorem}\label{thm:affine-support}
The Serre-ideal function $F$ is generically nonzero on a maximal $E$-filtered component of rank $(2,3,2)$. This component contains a dense orbit of a rigid, locally free module $M$ with $F(M)=-2$. Consequently, \cite[Conjecture~1.4]{GLSIV} fails even for the minimal symmetrizer of an affine Cartan matrix.
\end{theorem}
\begin{proof}
Set $R_0=\C[\varepsilon]/(\varepsilon^2)$ and use arrows
$a:1\to2,\quad b:2\to1,\quad c:2\to3,\quad d:3\to2$.
With our right-module convention and suitable orientation signs, the relations are
\[ ab=0,\qquad dc=0,\qquad \varepsilon(ba-cd)+(ba-cd)\varepsilon=0; \]
the edge commutation relations follow from $\varepsilon^2=0$. Take
\[ M_1=\C u_1\oplus\C u_2,\qquad M_2=pR_0\oplus qR_0\oplus rR_0,\qquad M_3=\C v_1\oplus\C v_2, \]
with nonzero arrow values
\begin{equation}\label{eq:affine-module}
\begin{aligned}
 u_1a&=q,&u_2a&=-q\varepsilon,\\
 pb&=u_1,&(p\varepsilon)b&=u_2,\\
 v_1d&=p,&v_2d&=r+p\varepsilon.
\end{aligned}
\end{equation}
These specify the $\C$-linear maps on vector-space bases; all other values are zero. Only $pba=q$ and $(p\varepsilon)ba=-q\varepsilon$ are nonzero, so the relations hold. The chain
\[ 0\subset(0,qR_0,0)\subset(M_1,qR_0,0)\subset(M_1,M_2,0)\subset M \]
has quotients $E_2,E_1^{\oplus2},E_2^{\oplus2},E_3^{\oplus2}$ and refines to an $E$-filtration.

In these bases, the endomorphism equations give exactly
\[ Q_1=Q_3=\begin{pmatrix}\alpha&0\\\beta&\alpha\end{pmatrix},\qquad P=\begin{pmatrix}\alpha+\beta\varepsilon&0&0\\0&\alpha-\beta\varepsilon&0\\\beta&0&\alpha-\beta\varepsilon\end{pmatrix},\qquad \alpha,\beta\in\C. \]
Indeed, $a,b$ determine $Q_1$, the first row and second column of $P$; $d$ determines the remaining four entries and $Q_3$. The choices $(\alpha,\beta)=(1,0),(0,1)$ give the identity and a nonzero square-zero endomorphism. Thus $\dim\End(M)=2$, and the orbit has dimension $44-2=42$, attaining the bound in \cite[Theorem~1.1]{GLSIV}. The Hom--Ext formula \cite[Lemma~2.1]{GLSIV} gives $\dim\Ext^1(M,M)=2\dim\End(M)-2q_{d_*}(\mathbf r)=0$. With the loop fixed, the orbit has dimension $26-2=24=N_{d_*}(\mathbf r)$; denote its closure by $Z$.

Write words as strings of vertex labels and set
\[ w_+=2211323,\qquad w_0=2121323,\qquad w_-=1221323. \]
Then $F(M)=\chi_c\Fl^{\lf}_{w_+}(M)-2\chi_c\Fl^{\lf}_{w_0}(M)+\chi_c\Fl^{\lf}_{w_-}(M)$. The word $w_-$ is impossible because $a$ is injective. For $w_+$, the first two factors force the middle submodule to be
\[ \mc{K}=\ker(\cdot b)=qR_0\oplus rR_0. \]
The later vertex-$3$ line would map into $\mc{K}$, but $M_3d\cap\mc{K}=0$. Thus this flag scheme is also empty.

For $w_0$, let $U,N,L_3$ be the vertex-$1$ line after step two, the middle plane after step three, and the vertex-$3$ line after step five. Stability requires
\[ qR_0\subseteq N,\qquad Nb\subseteq U,\qquad L_3d\subseteq N. \]
The identity $(v\varepsilon)b=\nu(vb)$, where $\nu(u_1)=u_2$ and $\nu(u_2)=0$, makes $Nb$ invariant under $\nu$. It is nonzero, since otherwise $N=\mc{K}$. Hence $Nb=U=\C u_2$. Modulo $qR_0$, a primitive generator of $N$ normalizes to $r+\lambda p\varepsilon$, giving
\[ N_\lambda=qR_0\oplus(r+\lambda p\varepsilon)R_0,\qquad \lambda\in\C^\times. \]
The condition on $L_3$ forces $\lambda=1$ and $L_3=\C v_2$. The first middle line must be killed by $b$ and contain $-q\varepsilon$; its possibilities are exactly
\[ L_z=(q+zr\varepsilon)R_0,\qquad z\in\C. \]
These choices determine the entire flag, and each $z$ gives free successive quotients of the prescribed ranks. The parametrization and its inverse are regular on this graph chart. Thus
\begin{equation}\label{eq:affine-three-flags}
 \Fl^{\lf}_{w_+}(M)=\varnothing,\qquad
 (\Fl^{\lf}_{w_0}(M))_{\mr{red}}\cong\A^1_\C,\qquad
 \Fl^{\lf}_{w_-}(M)=\varnothing.
\end{equation}
Thus $F$ equals $-2$ on the dense orbit of $M$. Its support has dimension $42=\dim G(\mathbf v)-q_{d_*}(\mathbf r)$, contradicting the strict bound in \cite[Conjecture~1.4]{GLSIV}.
\end{proof}

\begin{proposition}\label{prop:primitive-support}
Let $\Theta=[[[[S_{21},T_1],T_3],T_2],T_3]$ and $\Psi=-\tfrac12\Theta$. Then $\Psi\in\mc{I}$ is primitive and $\Psi(M)=1$. Thus the support bound in \cite[Conjecture~1.4]{GLSIV} fails even for primitive elements of the Serre ideal.
\end{proposition}
\begin{proof}
Every bracket preserves primitivity and membership in $\mc{I}$. A contributing word must begin with $2$ and end with $3$: $a,d$ are injective, $b$ is surjective, and $(M_1a)R_0+(M_3d)R_0=M_2$. A prefix $23$ is excluded by $M_3d\cap\mc{K}=0$. Of the remaining words, $2211323$ was excluded above, while $2212133$ and $2221133$ insert all of $M_2$ before all of $M_1$. The remaining terms are
\begin{center}
\begin{tabular}{lrrr}
\toprule
Word $w$&$2112233$&$2121233$&$2121323$\\
\midrule
Coefficient in $\Theta$&$1$&$-2$&$-2$\\
$\chi_c\Fl^{\lf}_w(M)$&$8$&$4$&$1$\\
\bottomrule
\end{tabular}
\end{center}
For $2112233$, the first middle line is $qR_0$; flags in $M_1,M_3$ and a free line in $R_0^2$ each contribute Euler characteristic $2$. For $2121233$, use the notation $L,U,N$ above. If $N=\mc{K}$, the cases $U\ne\C u_2$ and $U=\C u_2$ each form an affine line. Otherwise $U=\C u_2$, $N=N_\lambda$ and $L=L_z$, giving $\C^\times\times\A^1$. The final flag in $M_3$ contributes $2$, so the value is $(1+1+0)\cdot2=4$. The last entry is \eqref{eq:affine-three-flags}. Hence $\Theta(M)=8-8-2=-2$.
\end{proof}
\noindent This is consistent with \cite[Corollary~6.5]{GLSIV}, which establishes the support bound for the individual Serre generators.

For $Y\in\Bcal_{d_*}(\mathbf s)$, let $\rho_Y:\mc{A}_{\mathbf s}\to\Q$ be generic evaluation, so $\lambda_Y^{\mr{sp}}=\rho_Y\ev_\Pi$. Put
\[ \mc{N}_{\mr{sp},\mathbf s}=\bigcap_{Y\in\Bcal_{d_*}(\mathbf s)}\ker\rho_Y,\qquad \mc{N}_{\mr{sp}}=\bigoplus_{\mathbf s}\mc{N}_{\mr{sp},\mathbf s}. \]

\begin{corollary}\label{cor:generic-kernel-not-ideal}
The subspace $\mc{N}_{\mr{sp}}$ is not a right ideal. In degree $\mathbf r=(2,3,2)$, neither $\mc{I}_{\mathbf r}$ nor $\mc{N}_{\mr{sp},\mathbf r}$ contains the other. Special-fibre generic values on maximal components therefore do not determine a class in $\mc{A}/\mc{I}$.
\end{corollary}
\begin{proof}
By \cite[Corollary~6.5]{GLSIV}, $S_{21}\in\mc{N}_{\mr{sp}}$, whereas $S_{21}T_1T_3T_2T_3=F\in\mc{I}$ and $\rho_Z(F)=-2$. Thus $\mc{N}_{\mr{sp}}$ is not a right ideal and $\mc{I}_{\mathbf r}\nsubseteq\mc{N}_{\mr{sp},\mathbf r}$.

By \cite[Theorem~7.3]{GLSIV}, the generic-evaluation matrix of convolution words has full row rank over $\C$. Its entries are integers, so it has the same rank over $\Q$. Together with \eqref{eq:published} and Theorem~\ref{thm:canonical-quotient}, this gives
\[ \op{codim}_{\mc{A}_{\mathbf s}}\mc{I}_{\mathbf s}
 =\dim_\Q U^+_{\Q,\mathbf s}
 =|\Bcal_{d_*}(\mathbf s)|
 =\op{codim}_{\mc{A}_{\mathbf s}}\mc{N}_{\mr{sp},\mathbf s}. \]
The two subspaces in degree $\mathbf r$ are distinct and have equal codimension, so neither contains the other. Adding an element of $\mc{N}_{\mr{sp},\mathbf r}\setminus\mc{I}_{\mathbf r}$ to a normalized lift preserves its generic values but changes its quotient class.
\end{proof}

\subsection{The ramified lift and monodromy}\label{sec:affine-monodromy}
Put $R=\OO[x]/(x^2-t)$. On vertex lattices $\OO^2,R^3,\OO^2$, define the arrow actions by \eqref{eq:affine-module}, replacing $\varepsilon$ by $x$; specify $b$ on the $\OO$-basis $p,px,q,qx,r,rx$. Denote the resulting lattice by $\ms{M}$. The relation $x(ba-cd)+(ba-cd)x=0$ holds on $p,px$ by cancellation; all other relations and the filtration above persist. Thus $\ms{M}$ lifts $M$ and has nilpotent generic fibre.

Let $\mf{F}_w(\ms{M})$ be the relative flag scheme over $\OO$. Over an $\OO$-algebra $B$, its flags are locally direct summands over the vertex rings $B,R_B,B$, where $R_B=B[x]/(x^2-t)$, with locally free successive quotients of the ranks prescribed by $w$. Its generic fibre is the balanced flag scheme of $\ms{M}_K$.

\begin{proposition}\label{prop:affine-relative}
The schemes $\mf{F}_{w_+}(\ms{M})$ and $\mf{F}_{w_-}(\ms{M})$ are empty, and
\begin{equation}\label{eq:affine-relative}
 \mf{F}_{w_0}(\ms{M})\cong\op{Spec}\OO[\lambda,z]/(\lambda^2-t,\lambda z).
\end{equation}
Its reduced special fibre is $\A^1_\C$ and its generic fibre is $\op{Spec}K(\sqrt t)$.
\end{proposition}
\begin{proof}
The exclusions of $w_+,w_-$ hold in both geometric fibres. For $w_0$, use $L,U,N,L_3$ as above. The line $L_3$ cannot be spanned by $v_1$ at a geometric point: then $pR_B\subseteq N$ would force $Nb=M_1$. Thus the chart
$L_3=B(\lambda v_1+v_2)$ covers the scheme. Since $N$ contains $qR_B$ and $L_3d$, it equals
\[ N=qR_B\oplus nR_B,\qquad n=r+p(\lambda+x). \]
Now $nb=\lambda u_1+u_2$ forces $U=B(\lambda u_1+u_2)$, and $(nx)b=tu_1+\lambda u_2\in U$ is equivalent to $\lambda^2=t$.

The projection $L\to qR_B$ along $nR_B$ is an isomorphism at every geometric point. On the special fibre this follows from $L=L_z$. On the generic fibre, $Ua\subseteq L$ forces the $q$-line in the $x=-\lambda$ eigenspace, and $Lb=0$ forces it in the $x=\lambda$ eigenspace. Hence the graph chart covers the scheme, and we may write
$L=(q+n(h_0+xh_1))R_B$. The conditions $Lb=0$ and $Ua\subseteq L$ give, respectively,
\[ \begin{aligned}
 h_0+\lambda h_1&=0,&\lambda h_0+th_1&=0,\\
 \lambda h_0-th_1&=0,&\lambda h_1-h_0&=0.
\end{aligned} \]
Together with $\lambda^2=t$, these are equivalent to $h_0=0$ and $\lambda h_1=0$, since $2$ is invertible. Setting $z=h_1$ gives regular coordinates $\lambda,z$ that reconstruct each flag over $B$, proving the scheme isomorphism \eqref{eq:affine-relative}.
\end{proof}

The relative scheme is not flat: $tz=0$ although $z\ne0$. Its special fibre is $\op{Spec}\C[\lambda,z]/(\lambda^2,\lambda z)$; the reduction is $\A^1$ and the nilradical $(\lambda)$ is supported at the origin. Its Euler characteristic is $1$. The two geometric generic points are exchanged by tame inertia; their cohomology is the induced representation $H^0\cong\op{Ind}_{G_{K(\sqrt t)}}^{G_K}\Q_\ell$,
whose monodromy trace is zero. Thus $\ev_\sigma(f)(\ms{M}_K)=0$, although the geometric generic-fibre Euler characteristic is $2$.

To identify the component-generic values, put $\ms{M}_k=\ms{M}/t^k\ms{M}$. Solving the arrow-intertwining equations over $\OO/(t^k)$, without inverting $t$, gives
\[ Q_1=Q_3=\begin{pmatrix}\alpha&t\beta\\\beta&\alpha\end{pmatrix},\qquad P=\begin{pmatrix}\alpha+\beta x&0&0\\0&\alpha-\beta x&0\\\beta&0&\alpha-\beta x\end{pmatrix},\qquad \alpha,\beta\in\OO/t^k. \]
Hence $\dim_\C\End_{\Pi_k}(\ms{M}_k)=2k$ and $\dim G_k(\mathbf r)=26k$. The orbit closures $Z_k$ are maximal of dimension $24k$, and correspond to $Z$ under reduction. Applying a vertex-basis change with algebraically independent power-series coefficients places every reduction at the generic point of its dense orbit without changing the flag schemes. Therefore
\[ \lambda_Z^{\mr{sp}}(f)=-2\ \text{ but }\ \lambda_Z^{\mon}(f)=0, \]
proving Theorem~\ref{intro:counterexample}.

For $k\ge2$, the three ordinary flag varieties of $\ms{M}_k$ are empty: the middle one would require $\lambda^2=t$ in $\C[t]/(t^k)$, impossible by comparing constant and linear coefficients. Thus ordinary generic values can change under common scaling of the symmetrizer, while the canonical Serre quotient remains $U^+_\Q$.  The monodromy functional, unlike the special-fibre one, descends to this quotient.

\appendix

\section{The acyclic realization}\label{sec:acyclic}
We return to arbitrary $(C,d)$ and the acyclic orientation $\Omega$ fixed in Section~\ref{sec:setup}. We give a proof of the acyclic realization that does not use the component charts. Let $H=H(C,D,\Omega)^{\mr{op}}$, and let $\mc{M}_\Q(H)$ be the rational convolution algebra generated by the characteristic functions $\theta_i=1_{E_i}$ of its locally free vertex-supported modules. Thus $\mc{M}_\Q(H)$ is the rational form of the algebra in \cite{GLSIII}. Let $\Lambda_d$ be the quotient of $\Gamma_d$ obtained by setting the arrows opposite to $\Omega$ equal to zero. We write $\mc{M}_\sigma(\Lambda_d)$ for the rational span of its balanced flag-trace functions, as for $\Gamma_d$.

\begin{theorem}\label{thm:acyclic}
The assignments $e_i\mapsto\theta_i$ and $e_i\mapsto F_i$ give Hopf algebra isomorphisms
\[ U^+_\Q\xrightarrow{\sim}\mc{M}_\Q(H),\qquad U^+_\Q\xrightarrow{\sim}\mc{M}_\sigma(\Lambda_d). \]
\end{theorem}
The first map is surjective by \cite[Corollary~4.10]{GLSIII}; its injectivity is \cite[Conjecture~7.2]{GLSIII}. We prove both injectivity statements by comparing flag evaluations with those for an ordinary acyclic quiver.

\subsection{An auxiliary quiver and its Lie algebra}
Put $L=\lcm_i d_i$, $w_i=L/d_i$, and $h_{ij}=L/e_{ij}=\lcm(w_i,w_j)$ for adjacent $i,j$. Define a quiver $Q^{\mr{c}}$ with vertex set
\[ I^{\mr{c}}=\{(i,p):i\in I,\ 0\le p<w_i\},\qquad \pi:I^{\mr{c}}\to I,\quad\pi(i,p)=i. \]
Write $I_i^{\mr{c}}=\pi^{-1}(i)$. For every arrow $a=a_{ji}^{(c)}:j\to i$ that survives in the acyclic quotient, and every $0\le r<h_{ij}$, put an arrow
$a_r:(j,r\bmod w_j)\longrightarrow(i,r\bmod w_i)$.
This is done separately for each $1\le c\le g_{ij}$, so all parallel arrows are included. A directed cycle would project to one in $\Omega$; hence $Q^{\mr{c}}$ is acyclic.

Let $\widetilde C=(\widetilde c_{uv})_{u,v\in I^{\mr{c}}}$ be its symmetric Cartan matrix: the diagonal entries are $2$, and for $u\ne v$, $-\widetilde c_{uv}$ is the number of arrows joining $u$ and $v$. There are no arrows between vertices in a single $I_i^{\mr{c}}$, and
\begin{equation}\label{eq:cover-column}
 \sum_{u\in I_i^{\mr{c}}}\widetilde c_{uv}=c_{ij}\qquad(v\in I_j^{\mr{c}}).
\end{equation}
For $i\ne j$, a fixed vertex above $j$ is incident with $h_{ij}/w_j=f_{ij}$ arrows above each original arrow. Multiplication by $g_{ij}$ gives $-c_{ij}$. For $i=j$, the sum is $2$. The fibre sizes here are $L/d_i$, rather than the $d_i$ vertices obtained by splitting the generic species in Section~\ref{sec:setup}.

Denote the Chevalley generators and simple coroots of $\mf{g}(\widetilde C)$ by $\widetilde e_u,\widetilde f_u,\widetilde h_u$. Put
\[ E_i'=\sum_{u\in I_i^{\mr{c}}}\widetilde e_u,\qquad F_i'=\sum_{u\in I_i^{\mr{c}}}\widetilde f_u,\qquad H_i'=\sum_{u\in I_i^{\mr{c}}}\widetilde h_u. \]
\begin{lemma}\label{lem:Lie-injection}
The assignment $e_i\mapsto E_i'$ gives an injection $\mf{n}(C)\hookrightarrow\mf{n}(\widetilde C)$, and therefore an injection of their enveloping algebras.
\end{lemma}
\begin{proof}
Equation~\eqref{eq:cover-column} gives the required Cartan and Chevalley commutators for $E_i',F_i',H_i'$. The operators $\op{ad}\widetilde e_u$, for $u\in I_i^{\mr{c}}$, commute. Fix $i\ne j$ and $v\in I_j^{\mr{c}}$. In the expansion of $(\op{ad}E_i')^{1-c_{ij}}\widetilde e_v$, a term has total degree
$1-c_{ij}=1+\sum_{u\in I_i^{\mr{c}}}(-\widetilde c_{uv})$.
At least one factor $\op{ad}\widetilde e_u$ therefore occurs more than $-\widetilde c_{uv}$ times. Since the operators commute, the corresponding Serre relation makes the term zero. This proves the positive Serre identities; the negative identities follow in the same way.

To obtain injectivity also when $C$ is singular, use minimal Cartan realizations $\mf{h}_C$ and $\mf{h}_{\widetilde C}$ \cite[Chapters~1 and~9]{Kac}. In such a realization the simple roots and simple coroots are each linearly independent, and $\dim\mf{h}_C=2|I|-\op{rank}C$. Here we also regard each $\alpha_j$ as the corresponding linear form on $\mf{h}_C$, and write $\widetilde\alpha_v$ for the target simple roots. Extend $h_i\mapsto H_i'$ to a linear map $\psi:\mf{h}_C\to\mf{h}_{\widetilde C}$ satisfying
$\widetilde\alpha_v(\psi(h))=\alpha_{\pi(v)}(h)$ where $h\in\mf{h}_C$ and $v\in I^{\mr{c}}$.
These conditions agree on the coroot span by \eqref{eq:cover-column}. They can be prescribed on a complementary basis because the target simple roots are independent.

If $\psi(h)=0$, then every $\alpha_j(h)=0$. In a minimal realization, this common kernel is contained in the simple-coroot span: both spaces of solutions, in $\mf{h}_C$ and in that span, have dimension $|I|-\op{rank}C$. The vectors $H_i'$ are independent, since their supports in the independent target coroots are disjoint. Hence $h=0$.

The Chevalley--Serre presentation now gives a homomorphism $\mf{g}(C)\to\mf{g}(\widetilde C)$ injective on the Cartan subalgebra. By the defining maximal-ideal property of a contragredient Lie algebra, its kernel is zero: no nonzero ideal has zero intersection with the Cartan subalgebra \cite[Chapters~1 and~9]{Kac}. Restriction to the positive parts proves the first assertion. The Poincar\'e--Birkhoff--Witt theorem gives the assertion about enveloping algebras.
\end{proof}

\subsection{Integral representations and ordinary flag evaluations}
Let $V$ be a finite-dimensional representation of $Q^{\mr{c}}$ over $\C$. Write $V_{i,p}$ for its vertex spaces and $V(a_r)$ for its arrow maps. Define
\[ \ms{N}(V)_i=\bigoplus_{p=0}^{w_i-1}V_{i,p}\otimes_\C\OO_i. \]
Its rank over $\OO_i$ is $\sum_p\dim_\C V_{i,p}$. Decompose each vertex lattice into the $\OO$-summands $V_{i,p}\otimes x_i^b\OO$, indexed by
\[ r=p+bw_i,\qquad 0\le p<w_i,\quad 0\le b<d_i,\quad 0\le r<L. \]
Multiplication by $x_i$ increases $r$ by $w_i$, with a factor $t$ whenever the index passes $L-1$. For a surviving arrow $a:j\to i$, its action from the summand indexed by $r$ at $j$ to the summand with the same index at $i$ is $V(a_{r\bmod h_{ij}})$, extended $\OO$-linearly. All opposite arrows act as zero.

Both sides of the edge relation \eqref{eq:comm} increase the index by $h_{ij}$ and have the same factor of $t$ if the index passes $L-1$. The arrow matrix is unchanged by this increase, because its index is taken modulo $h_{ij}$. The moment relations hold because the opposite arrows are zero. Thus these formulas define a locally free $\PP_d$-lattice, with nilpotent generic fibre by acyclicity. Put
\[ \Phi(V)=\ms{N}(V)/t\ms{N}(V). \]

Fix a word $\mathbf i=(i_1,\ldots,i_h)$ whose rank equals that of $\ms{N}(V)$. Give the summand indexed by $r$ weight $r$ for an action of $\Gm$. This action commutes with the arrow maps and conjugates each loop map to a nonzero scalar multiple of itself, so it preserves the locally free flag varieties of $\Phi(V)$. A submodule fixed by this action is a direct sum of its weight subspaces. At vertex $i$, loop stability describes those subspaces by chains
$U_{p,0}\subseteq\cdots\subseteq U_{p,d_i-1}\subseteq V_{i,p}\quad(0\le p<w_i)$.
If this submodule is free over $H_{i,1}$, its dimension is $d_i$ times the dimension of the kernel of $x_i$ on it. The latter kernel is the sum of the last terms of the chains. Therefore
\[ \sum_{p,b}\dim U_{p,b}=d_i\sum_p\dim U_{p,d_i-1}. \]
Every inclusion in each chain must be an equality. Conversely, constant chains give free submodules with free quotients. Their arrow conditions are precisely the subrepresentation conditions for $V$.

Let $\Fl_{\mathbf v}(V)$ denote the ordinary representation flag variety with one-dimensional simple successive quotients at the vertices listed in the word $\mathbf v$ of $I^{\mr{c}}$. Extend $\pi$ to words letter by letter. Each successive quotient $E_i$ in a fixed flag of $\Phi(V)$ chooses one vertex in $I_i^{\mr{c}}$. Thus its fixed locus is the disjoint union of the $\Fl_{\mathbf v}(V)$ with $\pi(\mathbf v)=\mathbf i$.

Stratify the nonfixed locus into invariant locally closed pieces admitting geometric quotients. Their fibres are $\Gm/\mu_n$, with Euler characteristic zero, so Euler integration gives equality with the Euler characteristic of the fixed locus. We obtain
\begin{equation}\label{eq:acyclic-detection}
 \chi_c\Fl^{\lf}_{\mathbf i}(\Phi(V))=\sum_{\pi(\mathbf v)=\mathbf i}\chi_c\Fl_{\mathbf v}(V).
\end{equation}

\subsection{Generic flag traces and faithfulness}
The matrices of $\ms{N}(V)$ are defined over $\C[t]$. After the base change $t=u^L$, multiply the basis vectors in the summand indexed by $r$ by $u^{-r}$. The loop matrix becomes $u^{w_i}$ times its value at $t=1$, and each arrow matrix becomes its value at $t=1$. These nonzero scalar factors do not change invariant subspaces.

Let $\mc{F}_{\mathbf i}(V)$ be the complex projective variety of flags stable under the matrices at $t=1$ and satisfying the balanced dimension conditions. The preceding change of basis identifies the generic flag scheme, over $\C((u))$, with the scalar extension of $\mc{F}_{\mathbf i}(V)$. Write $\sigma(u)=\zeta u$ for the chosen tame generator; then $\zeta\in\C$ is a primitive $L$th root of unity. In these coordinates its action on the constant flag variety is induced by $\tau(v)=\zeta^{-r}v$ (on the summand indexed by $r$).
The distinct eigenvalues of $\tau$ force a fixed flag to be homogeneous in $r$. The loop at $t=1$ is invertible and cyclically permutes the summands with a fixed residue modulo $w_i$. Its invariant flags therefore have equal subspaces in all those summands. They are balanced and correspond exactly to the ordinary flags on the right side of \eqref{eq:acyclic-detection}.

For a finite-order automorphism of a complex projective variety, an equivariant triangulation gives the Lefschetz formula: after subdivision, only pointwise fixed simplices contribute to the alternating trace. Hence this trace is the Euler characteristic of the fixed locus. Applying this formula to $\tau$ and using comparison with \'etale cohomology gives
\begin{equation}\label{eq:acyclic-trace}
 F_{\mathbf i}(\ms{N}(V)_K)=\sum_{\pi(\mathbf v)=\mathbf i}\chi_c\Fl_{\mathbf v}(V).
\end{equation}

\begin{proof}[Proof of Theorem~\ref{thm:acyclic}]
Let $\widetilde\theta_v$ be the characteristic function of the simple representation at $v\in I^{\mr{c}}$ in the ordinary-quiver convolution algebra. Schofield's ordinary-quiver realization, recalled in \cite[Introduction and Section~7]{GLSIII}, identifies the composition algebra with $U_\Q(\mf n(\widetilde C))$.  By Lemma~\ref{lem:Lie-injection}, the algebra homomorphism from $T_\Q$ to the ordinary-quiver composition algebra, sending $z_i$ to $\sum_{v\in I_i^{\mr{c}}}\widetilde\theta_v$, has kernel exactly $J_C$.

If a word polynomial evaluates to zero on every $H$-module, it evaluates to zero on every $\Phi(V)$. Equation~\eqref{eq:acyclic-detection} shows that its image under this homomorphism is zero, so it belongs to $J_C$. The opposite inclusion is the acyclic Serre identity of \cite[Proposition~4.9]{GLSIII}. This proves the first isomorphism.

For $\Lambda_d$, \eqref{eq:acyclic-trace} gives the same reverse inclusion of kernels. The forward inclusion is Corollary~\ref{thm:Serre}, applied to modules with all opposite arrows zero. Thus both evaluation kernels are $J_C$. The direct-sum flag identities identify the coproducts with the usual coproduct on $U^+_\Q$, proving both Hopf isomorphisms. Since the same generic modules are $\Gamma_d$-modules, they also prove faithfulness for $\Gamma_d$ independently of the component construction.
\end{proof}

The ordinary isomorphism proved here is used in the proof of Theorem~\ref{thm:canonical-quotient}, which identifies restriction to $H$ with the canonical Serre quotient of preprojective convolution.

\section*{Declaration of AI-assisted research and manuscript preparation}
The author used ChatGPT 6.0 during the development of this
work to explore examples, suggest and examine possible proof strategies,
assist with computational checks, and improve the exposition and \LaTeX{}
presentation. The author takes sole responsibility for all mathematical
claims, proofs, references, and the final content of the manuscript.

\end{document}